\documentclass[english]{article}
\usepackage{a4wide}
\usepackage[utf8]{inputenc}
\usepackage[T1]{fontenc}
\usepackage{babel}
\usepackage{color}
\usepackage{amsmath, amsfonts, amsthm}
\usepackage{graphicx}
\usepackage[utf8]{inputenc}

\usepackage{amsmath,amssymb,cases,color}
\usepackage{hyperref}
\usepackage[capitalise]{cleveref}
\usepackage{vmargin}
\setmarginsrb{1.5cm}{1cm}{1.5cm}{3cm}{1cm}{1cm}{2cm}{2cm}
\usepackage{dsfont}

\usepackage{amsmath, amssymb,amscd, }
\usepackage{amsfonts}
\usepackage{mathrsfs}
\usepackage{graphicx}

 \usepackage{upref}
\hypersetup{linkcolor=blue, colorlinks=true,citecolor = red}
\usepackage{bm}

\newtheorem{Theorem}{Theorem}[section]
\newtheorem{Proposition}{Proposition}[section]
\newtheorem{Lemma}{Lemma}[section]
\newtheorem{Remark}{Remark}[section]
\newtheorem{Corollary}{Corollary}[section]
\newtheorem{Definition}{Definition}[section]
\numberwithin{equation}{section}

\newcommand{\bTheorem}[1]{
\begin{Theorem} \label{T#1} }
\newcommand{\eT}{\end{Theorem}}

\newcommand{\bProposition}[1]{
\begin{Proposition} \label{P#1}}
\newcommand{\eP}{\end{Proposition}}

\newcommand{\bLemma}[1]{
\begin{Lemma} \label{L#1} }
\newcommand{\eL}{\end{Lemma}}

\newcommand{\bCorollary}[1]{
\begin{Corollary} \label{C#1} }
\newcommand{\eC}{\end{Corollary}}
\newcommand{\bFormula}[1]{
\begin{equation} \label{#1}}
\newcommand{\eF}{\end{equation}}

\newcommand{\Div}       {{\rm div}}

\newcommand{\dx}        {\,{\rm d}\x}
\newcommand{\dt}        {\,{\rm d}t}

\newcommand{\D}     {{\mathcal D}}
\newcommand{\B}     {{\mathcal B}}
\newcommand{\F}     {{\mathcal F}}

\newcommand{\bY} {{\bf Y}}
\newcommand{\bZ} {{\bf Z}}
\newcommand{\bU} {{\bf U}}

\newcommand{\vr}        {\varrho}

\newcommand{\vu}        {\vc{u}}
\newcommand{\vphi}  {\pmb{\varphi}}

\newcommand{\Om}{\Omega}

\newcommand{\vn}{\vc{n}}
\newcommand{\vU}{\vc{U}}
\newcommand{\vw}{\vc{w}}

\newcommand{\vV}{\vc{V}}
\newcommand{\dv}{{\rm div}}
\newcommand{\dS}{{\rm d} {S}}

\renewcommand{\d}{{\rm d} }

\newcommand{\R}{\mathbb R}

\newcommand{\de}{\partial}

\def\bU{{\bf U}}

\def\bY{{\bf Y}}

\newcommand{\vc}[1]     { {\bf #1} }
\newcommand{\tn}[1]     {\mathbb{#1} }
\newcommand{\x}{{\mathbf x}}
\newcommand{\X}{{\mathbf X}}
\newcommand{\y}{{\mathbf y}}

\newcommand{\ZZ}{\widetilde{\bZ}_2}
\newcommand{\ZZj}{\widetilde{\bZ}_1}

\newcommand{\Ot}{\widetilde{\tn{O}}}

\newcommand{\mc}{\mathcal}

\begin{document}

\title{\bf Measure-valued solutions and weak-strong uniqueness for the incompressible inviscid fluid-rigid body interaction}

\author{
${^1}$Matteo Caggio, ${^2}$Ond\v rej Kreml, ${^2}$\v S\'arka Ne\v casov\'a\textsuperscript{*}, ${^2}$Arnab Roy, ${^3}$Tong Tang
}
\date{}

\maketitle

\begin{center}
\textsuperscript{1}
Department of Mathematics, Faculty of Science, University of Zagreb \\
Bijeni\v cka cesta 30, 10000 Zagreb, Croatia. \\
\end{center}

\begin{center}
\textsuperscript{2}
Institute of Mathematics, Czech Academy of Sciences \\
\v Zitn\'a 25, 115 67 Praha 1, Czech Republic.\\
\end{center}

\begin{center}
\textsuperscript{3}
Department of Mathematics, College of Sciences, \\
Hohai University, Nanjing 210098, P.R. China.
\end{center}

\begin{center}
\textsuperscript{*}Corresponding author
\end{center}

\vspace{1cm}

\begin{abstract}
We consider a coupled system of partial and ordinary differential equations describing the interaction between an isentropic inviscid fluid and a rigid body moving freely inside the fluid. We prove the existence of measure-valued solutions which is generated by the vanishing viscosity limit of incompressible fluid-rigid body interaction system under some physically constitutive relations. Moreover, we show that the measure-value solution coincides with strong solution on the interval of its existence. This relies on the weak-strong uniqueness analysis.
\end{abstract}

{\bf Key words.} Euler equations, fluid–rigid body interaction, measure-valued solutions, weak-strong uniqueness.
\bigskip

 {\bf AMS subject classifications.} 35Q35, 35Q31, 35R37, 76B99.


\section{Introduction}
We consider the motion of a rigid body inside an isentropic inviscid fluid. The fluid and the body occupy a bounded domain $\Omega\subset \R^{3}$. At the time $t\in \R^{+}$, we denote by $\mathcal{B}(t)\subset\Omega$ the bounded domain occupied by the rigid body and by $\mathcal{F}(t)=\Omega\setminus \overline{\mathcal{B}(t)}$, the domain filled by the fluid. Assuming that the initial position $\mathcal{B}(0)$ of the rigid body is prescribed, we denote $\mathcal{B}_{0}=\mathcal{B}(0)$ and, similarly, $\mathcal{F}_{0}=\mathcal{F}(0)$. The interface between the body and the fluid is denoted
by $\partial \mathcal{B}(t)$ and the normal vector to the boundary is denoted by $\vn(t,\x)$ pointing outside $\Omega$ and inside $\mathcal{B}(t)$. For $T > 0$ we introduce the following notation for the space-time cylinders
\begin{equation} \label{domain}
\begin{array}{l}
{Q}_{\mathcal{F}}= \bigcup_{t\in (0,T)} \{t\} \times \mathcal{F}(t),

\\
{Q}_{\partial \mathcal{B}}=\bigcup_{t\in (0,T)} \{t\} \times \partial\mathcal{B}(t),

\\
{Q}_{\mathcal{B}}=\bigcup_{t\in (0,T)} \{t\} \times \mathcal{B}(t),

\\
Q_T= (0,T) \times \Omega.
\end{array}
\end{equation}
The fluid motion is governed by the following system of equations
\begin{equation}\label{fluidmotion:ineuler}
\left\{
\begin{array}{ccc}
\partial_{t} \vu_{\mathcal{F}} + (\vu_{\mathcal{F%
}} \cdot \nabla) \vu_{\mathcal{F}} + \nabla p_{\mc{F}} = 0,\quad \operatorname{div} \vu_{\mathcal{F}}=0 & \mathrm{in} & {Q}_{%
\mathcal{F}},\\
\vu_{\mathcal{F}} \cdot \vc{n} = 0 & \mathrm{on} & (0,T) \times \partial\Omega , \\
\vu_{\mathcal{F}} \cdot \vc{n} = \vu_{\mathcal{B}} \cdot \vc{n}
&\mathrm{on} & Q_{\partial
\mathcal{B}}, \\
\vu_{\mathcal{F}}(0)= \vu_{\F_{0}} & \mathrm{in} & \mathcal{F%
}_{0},\\
\end{array}%
\right.
\end{equation}%
where $\vu_{\mathcal{F}}$ denotes the velocity of the fluid, the scalar function $p_\F$ is the pressure, and $\vu_{\mathcal{B}}$ is the full velocity of the rigid body. We assume that the external body forces acting on the fluid are zero. The velocity of the rigid body is given by
\begin{equation}\label{eq:rigidbodyvelocity}
\vu_{\mathcal{B}}(t,\x)=\vV(t) + \vw(t)\times (\x-\X(t))
\end{equation}
for any $(t,\x) \in Q_\B$, where the translation velocity $\vV$ and the angular velocity $\vw$ of the body satisfy the following system of the equations
\begin{equation}\label{bodymotion:ineuler}
\left\{
\begin{array}{l}\displaystyle
m\frac{\d}{\dt}\vV(t) = \int_{\partial\mathcal{B}(t)} p_{\mc{F}} \tn{I} \cdot \vc{n} \d S\quad
 \text{ in } (0,T), \\ \displaystyle
\tn{J}(t)\frac{\d}{\dt}\vw(t) = \tn{J}(t)\vw(t)\times\vw(t) + \int_{\partial\mathcal{B}(t)} (\x-\X(t)) \times p_{\mc{F}} \tn{I} \cdot \vc{n} \d S\quad \text{ in } (0,T),\\
\vV(0)=\vV_{0},\qquad \vw(0)= \vw_{0}.
\end{array}
\right.
\end{equation}
Here, the mass of the body, the center of mass $\X$ and the inertial tensor $\tn{J}$ are respectively,

\begin{align}
    m &= \int_{\B(t)} \vr_\B(t, \x) \dx, \label{def:m} \\
    \X(t) &= \frac{1}{m} \int_{\B(t)} \vr_\B(t, \x)\x \dx, \label{def:X} \\
    \tn{J}(t)\vc{a}\cdot\vc{b} &= \int_{\B(t)} \vr_\B(t, \x) \left[\vc{a}\times (\x-\X(t))\right]\cdot\left[\vc{b}\times (\x-\X(t))\right] \dx, \label{def:J}
\end{align}
with $\vr_\B$ denoting the density of the body which is assumed to be smooth but does not have to be constant, i.e. throughout this paper we assume $\vr_\B(0,\x) \in C^1(\overline{\B_0})$ and $\vr_\B(0,\x) \geq c_0 > 0$ in $\overline{\B_0}$, which then implies $\vr_\B \in C^1(\overline{Q_\B})$, $\vr_\B \geq c_0 > 0$ in $\overline{Q_\B}$. Note that, without loss of generality, we can assume that the center of mass of the body is at the origin at time zero, namely $\X(0) = 0$.
The position of the body $\B(t)$ is given by a time-dependent family of isometries of $\R^3$ such that
\begin{equation}\label{eq:isom1}
\eta[t]:  \R^3 \to \R^3,\ \ \overline{\B(t)}= \eta[t](\overline{\B_0})\ \ \mbox{for}\ 0\leq t\leq T,
\end{equation}
where the mapping $\eta[t]$ satisfies
\begin{equation}\label{eq:isom2}
\eta[t](\x) = \X(t) + \tn{O}(t) \x,\ \ \tn{O}(t)\in SO(3).
\end{equation}
The velocity of the body $\vu_\B(t,\x)$ is then naturally related to the isometries $\eta[t]$ by
\begin{equation}\label{eq:QO}
    \vV(t)= \frac{\d}{\d t} \X(t),\ \ \tn{Q}(t) = \left(\frac{\d}{\d t}\tn{O}(t)\right)\left(\tn{O}(t)\right)^{-1}
\end{equation}
for a.e. $t \in (0,T)$, where $\tn{Q}$ is an antisymmetric matrix such that 
\begin{equation}\label{eq:Qw}
    \tn{Q}(t)(\x-\X(t)) = \vw(t) \times (\x-\X(t)).
\end{equation}
Relations \eqref{eq:QO}-\eqref{eq:Qw} are consequence of the fact that the body is transported by its velocity, hence the family of isometries $\eta[t]$ satisfies 
\begin{equation}\label{eq:etaODE}
    \frac{\d}{\d t}\eta[t](\x) = \vu_\B(t,\eta[t](\x)) \qquad \mbox{ with } \eta[0](\x) = \x.
\end{equation}
Finally, equation \eqref{eq:etaODE} implies that the indicator function $\mathds{1}_\B(t,\x)$ of the set $\B(t)$ as well as the density of the body $\vr_\B(t,\x)$ satisfy the following transport (and continuity) equations
\begin{equation}\label{eq:chiBeq}
    \partial_t \mathds{1}_\B + \vu_\B\cdot\nabla \mathds{1}_\B =  \partial_t \mathds{1}_\B + \Div (\mathds{1}_\B \vu_\B) = 0 \qquad \mbox{ on } \mathbb{R}^3,
\end{equation}
\begin{equation}\label{eq:vrBeq}
    \partial_t \vr_\B + \vu_\B\cdot\nabla \vr_\B =  \partial_t \vr_\B + \Div (\vr_\B \vu_\B) = 0 \qquad \mbox{ on } \B(t).
\end{equation}

\subsection{Discussion and Main result}
In this paper, our first contribution is the proof of the existence theorem for measure-valued solutions to the system \eqref{fluidmotion:ineuler}-\eqref{bodymotion:ineuler}. In our framework the measure-valued solution consists of the position of the body $\B$, a Young measure $Y_{t,\x}$ and a dissipation defect $\D$, which is a bounded function in time. We give the precise definition of a measure-valued solution in Section \ref{s:existence}, see Definition \ref{mvs}. By $L^p$, resp. $H^m$ we denote the Lebesgue resp. Sobolev spaces. By $X_{\sigma}$ we denote a function space $X$ with the additional property of consisting of divergence free functions. For simplicity of notation we omit writing differentials $\dx$ and $\dt$ in the integral formulas in the rest of the paper.

\begin{Theorem}\label{thm:measure-euler}
Suppose $\Omega$ and $\mc{B}_0 \subset \Omega$ are two regular bounded domains of $\R^3$ and let $T > 0$. Let $\vu_{\F_0} \in L^2(\F_0)$, $\vV_0,\vw_0 \in \R^3$, such that $\vu_{\mc{F}_0}\cdot \vc{n}=(\vV_0 + \vw_0 \times \x)\cdot \vc{n}$ on $\partial \B_0$ and let $Y_{0,\x} = \delta_{\vV_0 + \vw_0 \times \x}$ for $\x \in \B_0$ and $Y_{0,\x} = \delta_{\vu_{\F_0}}$ for $\x \in \F_0$. Then there exists a measure-valued solution $(\B,Y_{t,\x},\D)$ of the system \eqref{fluidmotion:ineuler}-\eqref{bodymotion:ineuler} on time interval $(0,T)$ with initial data $(\B_0, Y_{0,\x})$.
\end{Theorem}

Our second main theorem is the weak-strong uniqueness theorem for measure-valued solutions of system \eqref{fluidmotion:ineuler}-\eqref{bodymotion:ineuler}. Before stating this theorem we first recall the result proved in \cite[Theorem 1.3]{MT10} regarding the existence of strong solutions to the system of rigid body moving inside an incompressible inviscid fluid. We note that the strong solution consists of the position of the body $\B(t) \subset \Omega$, fluid velocity $\vu_\F(t,\x)$, fluid pressure $p_\F(t,\x)$, translation velocity of the body $\vV(t)$ and angular velocity of the body $\vw(t)$ such that the equations \eqref{fluidmotion:ineuler}-\eqref{bodymotion:ineuler} are satisfied pointwise.


\begin{Theorem}\label{thm:strong solution}
Let $m\geq 3$ be an integer and $\B_0 \subset \subset \Omega$. Let $\vV_0 \in \R^3$, $\vw_0 \in \R^3$ and $\vu_{\F_{0}} \in H^m(\F_0,\R^3)$ satisfy:
\begin{align}
    \operatorname{div}\vu_{\mathcal{F}_0}(\x)&=0,\quad \x\in \F_0, \label{cc1}\\
    \vu_{\mathcal{F}_0}(\x)\cdot \vc{n}(\x)&= (\vV_0 + \vw_0 \times \x)\cdot \vc{n}(\x),\quad   \x\in \partial\B_0,\label{cc2}\\
  \vu_{\mathcal{F}_0}(\x)\cdot \vc{n}(\x)&= 0,\quad   \x\in \partial\Omega.\label{cc3}
\end{align}
Then there exists $T_0>0$ such that the system \eqref{fluidmotion:ineuler}-\eqref{bodymotion:ineuler} admits a unique strong solution $(\B,\vu_\F,p_\F,\vV,\vw)$ on time interval $(0,T_0)$ with initial data $(\B_0, \vu_{\mc{F}_0}, \vV_0, \vw_0)$
\begin{align}\label{strongspace}
    \vV &\in C^1[0,T_0),\quad \vw \in C^1[0,T_0), \notag\\
\vu_{\mathcal{F}} &\in C([0,T_0);H^m(\F(t))) \cap C^1([0,T_0);H^{m-1}(\F(t))),\\
p_{\F} &\in C([0,T_0);N^{m+1}(\F(t))),\notag
\end{align}
where for any open set $\mathcal{O} \subset \R^3$,
\begin{equation*}
    N^m(\mathcal{O})=\left\{q\in H^m(\mathcal{O}) \mid \int_{\mathcal{O}} q(\x) = 0\right\}.
\end{equation*}
\end{Theorem}
Our main goal is to show the weak-strong uniqueness property for system \eqref{fluidmotion:ineuler}-\eqref{bodymotion:ineuler}. More precisely, we want to prove the following result:
\begin{Theorem}\label{thm:main result}
Suppose $\Omega$ and $\mc{B}_0 \subset \subset \Omega$ are two regular bounded domains of $\R^3$. Let $\vV_0 \in \R^3$, $\vw_0 \in \R^3$ and $\vu_{\F_{0}} \in H^m(\F_0,\R^3)$ satisfy the compatibility conditions \eqref{cc1}-\eqref{cc3}. 

Let $(\B_2,\vu_{2 \mathcal{F}}, p_{2 \F} , \vV_{2} , \vw_{2})$ be a strong solution to \eqref{fluidmotion:ineuler}-\eqref{bodymotion:ineuler} on $(0,T_0)$ given by \cref{thm:strong solution} satisfying \eqref{strongspace} emanating from the initial data given by $\B_0, \vu_{\F_0}, \vV_0, \vw_0$. Let $(\B_1,Y_{t,\x},\mc{D})$ be a measure-valued solution to the system \eqref{fluidmotion:ineuler}-\eqref{bodymotion:ineuler} on $(0,T_0)$ emanating from the same initial data, i.e. $Y_{0,\x} = \delta_{\vV_0 + \vw_0 \times \x}$ for $\x \in \B_0$ and $Y_{0,\x} = \delta_{\vu_{\F_0}}$ for $\x \in \F_0$.

Then for $t \in (0,T_0)$
\begin{equation*}
\B_1(t) = \B_2(t),\ \ \mathcal{D}(t) = 0,\ \  Y_{t,\x}=\delta_{\vu_{2 \mathcal{F}(t,\x)}} \mbox{ on } \F_2(t) \mbox{ and } Y_{t,\x} = \delta_{\vV_2(t) + \vw_2(t) \times (\x-\X_2(t))} \mbox{ on } \B_2(t).
\end{equation*}

\end{Theorem}


The existence theory of weak and strong solutions for systems describing the motion of rigid body in {\it  a viscous incompressible} fluid was studied by many authors. For introduction to the problem of fluid coupled with rigid body see \cite{G2}, \cite{SER3}. Let us mention that first results on the existence of weak solutions until first collision go back to the works of  Conca, Starovoitov and Tucsnak \cite{CST}, Desjardins and Esteban \cite{DEES1}, Gunzburger, Lee and Seregin \cite{GLSE}, Hoffman and Starovoitov \cite{HOST}. Further, the possibility of collision in case of weak solution has been done in the works of \cite{SST},  \cite{F3}.
Let us also mention existence results on strong solutions, see e.g. \cite{T}, \cite{Wa}, \cite{GGH13}.

The case of the motion of rigid body in {\it an inviscid incompressible}  fluid is more complex  and we can expect that all problems which appear just in fluid alone must appear also there. Let us mention what is known.   
The case of  a smooth initial data with finite kinetic energy a problem has been investigated, see \cite{13}.  The case of Yudovich-like solutions (with bounded vorticities) was studied by O. Glass and F. Sueur,  see \cite{5}. The study with initial vorticity of the fluid belonging to  a $L^p_c$ vorticity, $p>2$ and the index $c$ is used here and in the sequel for “compactly supported” can be found in work of O. Glass, C. Lacave, F. Sueur, see \cite{4}. These works provided the global existence of solutions.  The result of \cite{13} was extended to the case of a solid of arbitrary form for which rotation has to be taken into account, see \cite{12}. The works \cite {5} and \cite {4} deal with an arbitrary form as well. Furthermore let us stress that  in \cite{6} the case of an initial vorticity in $L^p_c$ with $p>1$ was studied  in order to achieve the investigation of solutions “à la DiPerna–Majda”, referring here to the seminal work \cite{2} in the case of a fluid alone.
A famous result by Delort \cite{1} about the two-dimensional incompressible Euler equations is the existence of weak solutions when the initial vorticity is a bounded Radon measure with distinguished sign and lies in the Sobolev space $H^{-1}$. In paper by  F. Sueur \cite{S}  he was interested in the case where a rigid body immersed in the fluid is moving under the action of the fluid pressure. They proved the existence of solutions "à la Delort" in a particular case with a mirror symmetry assumption. Let us also mention uniqueness result of solution "a la Yudovic type" by Glass and Sueur, see \cite{GlassSueur}.

The aim of our first main theorem is to follow DiPerna, Majda \cite{2} approach to get the existence of measure-valued solution of the coupled system. In our knowledge, the concept of measure-valued solution in the case of fluid-structure interaction problem is new in the literature. The idea is to view this "inviscid incompressible + rigid body" problem as a vanishing viscosity limit of "viscous incompressible + rigid body" problem.  

Second part of our paper is devoted to the so-called weak-strong uniqueness but in more general framework of {\it measure-valued sense}. The weak-strong uniqueness is based on the  concept of relative entropy (or energy) inequality introduced  already by Dafermos \cite{D}. Our motivation goes back to work of  Y. Brenier, C. De Lellis, L. Székelyhidi \cite {BDS},  S. Demoulini, D. M. A. Stuart, A. E. Tzavaras \cite{DST} or Wiedemann \cite{W}. Let us mention the references that deal with the weak-strong uniqueness result for fluid-structure interaction problem. A weak-strong uniqueness result
for a motion on rigid body in an incompressible fluid has been shown recently in 2D case, see \cite{bravin2018weak, GlassSueur}, and in 3D case see,  \cite{CNM}, \cite{MNR}. Similar result was proved for the case of rigid body inside a compressible viscous fluid by Kreml et al. \cite{KrNePi_2}. In a case of rigid body with a cavity filled by  incompressible fluid the weak-strong uniqueness was shown in \cite{Disser2016}. For an analogous result for a cavity filled by compressible fluid see \cite{GMN}. We also refer to \cite{Dobo}, \cite{KrNePi} for problems on moving domains. But in these previous mentioned references regarding weak-strong uniqueness for fluid-structure interaction problem, the authors have always taken the motion of rigid body in a viscous fluid. In this article, we want to explore the case of "inviscid incompressible + rigid body". Here we have established the uniqueness result by suitably defining the relative energy functional and by estimating it via the energy inequalities of strong and measure-valued solution.

The outline of the paper is as follows. In Section \ref{s:existence} we define the measure-valued solution of the system \eqref{fluidmotion:ineuler}-\eqref{bodymotion:ineuler}. In this section, we establish Theorem \ref{thm:measure-euler} via replacing inviscid fluid by viscous fluid (with suitable boundary conditions) and by proving the viscosity limit of the sequence of weak solutions to "viscous incompressible + rigid body" system converges to a measure-valued solution for the system \eqref{fluidmotion:ineuler}-\eqref{bodymotion:ineuler}. Section \ref{s:wsu} is devoted to the proof of Theorem \ref{thm:main result}. In this section, firstly, we have done a change of variable for strong solution so that we can compare it with the measure-valued solution in the same domain. Secondly, we derived the energy estimate for the transformed strong solution. Finally, we have established weak-strong uniqueness in  \cref{thm:main result} with the help of proper estimate of relative energy functional. In Section \ref{s:appendix} we provide more details about derivation of several identities which are useful throughout the proofs of our main theorems.

\section{Existence of measure-valued solutions}\label{s:existence}

We want to derive a suitable formulation for the measure-valued solution of the system \eqref{fluidmotion:ineuler}-\eqref{bodymotion:ineuler} with the help of appropriate test functions. Due to the presence of this Navier-slip boundary condition, the test functions will also be discontinuous across the interface of fluid-solid. Let us first introduce the space of rigid velocity fields:
\begin{equation*}
\mc{R}=\left\{ \vphi_{\mc{B}} \mid \vphi_{\mc{B}}(\x)=\ell+\omega \times \x,\mbox{ for some }\ell \in \mathbb{R}^3,\, \omega\in \mathbb{R}^3\right\}.
\end{equation*}
For any $T>0$, we define the test function space $V_{T}$ as:
\begin{equation*}
V_{T}=
\left\{\!\begin{aligned}
&\vphi \in C([0,T]; L^2_{\sigma}(\Omega)),\mbox{ there exists }\vphi_{\mc{F}}\in \mc{D}([0,T]; \mc{D}_{\sigma}(\Omega)),\, \vphi_{\mc{B}}\in \mc{D}([0,T]; \mc{R})\\ &\mbox{ such that }\vphi(t,\cdot)=\vphi_{\mc{F}}(t,\cdot)\mbox{ on }\mc{F}(t),\quad \vphi(t,\cdot)=\vphi_{\mc{B}}(t,\cdot)\mbox{ on }\mc{B}(t)\mbox{ with }\\ &\vphi_{\mc{F}}(t,\cdot)\cdot \vc{n} = \vphi_{\mc{B}}(t,\cdot)\cdot \vc{n} \mbox{ on }\partial\mc{B}(t),\ \vphi_{\mc{F}}\cdot \vc{n}=0 \mbox{ on }\partial\Omega\mbox{ for all }t\in [0,T]
\end{aligned}\right\}.
\end{equation*}

Now, let us mention some basics on Young measures. Let $L^\infty_{weak-*}(Q_T;\mathcal{P}(\R^3))$ be the space of essentially bounded weakly-$*$ measure maps $Y:Q_T\rightarrow \mathcal{P}(\R^3)$, $(t,\x)\mapsto Y_{t,\x}$, where the notation $\mathcal{P}(\R^3)$ denotes the space of probability measures on $\R^3$. By virtue of fundamental theorem on Young measures, there exists a subsequence of $\{\mathbf u^{\varepsilon}\}_{\varepsilon>0}$ and parameterized family of probability measures $\{Y_{t,\x}\}_{(t,\x)\in Q_T}$
\begin{eqnarray*}
[(t,\x)\mapsto Y_{t,\x}]\in L^\infty_{weak-*}(Q_T;\mathcal{P}(\R^3)),
\end{eqnarray*}
called \textit{Young measure} associated to the sequence
$\{\mathbf u^{\varepsilon}\}_{\varepsilon>0}$, such that a.a. $(t,\x)\in Q_T$
\begin{eqnarray*}
\langle Y_{t,\x}; G(\mathbf u)\rangle=\widehat{G(\mathbf u)}(t,\x) \ \mbox{for any} \ G\in C_c(\Omega), \ \mbox{and} \ \mbox{a.a.} \ (t,\x)\in Q_T,
\end{eqnarray*}
whenever
\begin{eqnarray*}
G(\mathbf u^{\varepsilon})\rightarrow\widehat{G(\mathbf u)}(t,\x) \ \mbox{weakly-}* \ \mbox{in} \ L^\infty(Q_T).
\end{eqnarray*}
Above, the hat over a function is intended as weak limit.
If $G\in C(\Omega)$ is such that
\begin{eqnarray*}
\int^T_0\int_\Omega |G(\mathbf u^{\varepsilon})|\leq C,
\end{eqnarray*}
then $G$ is $Y_{t,\x}$ - integrable for almost all $(t,\x)\in Q_T$ and
\begin{eqnarray*}
[(t,\x)\mapsto \langle Y_{t,\x}; G(\mathbf u)\rangle]\in L^1(Q_T),
\end{eqnarray*}
and
\begin{eqnarray*}
G(\mathbf u^{\varepsilon})\rightarrow\widehat{G(\mathbf u)}(t,\x) \ \mbox{weakly-}* \ \mbox{in} \ \mathcal{M}(Q_T).
\end{eqnarray*}
Here, $\mathcal{M}(Q_T)$ denotes the space of signed measures on $Q_T = (0,T) \times \Omega$. Note that the Young measure $[(t,\x)\mapsto \langle Y_{t,\x}; G(\mathbf u)\rangle]$ is a parameterized family of non-negative measures acting on the phase space $\R^3$, while $\widehat{G(\mathbf u)}(t,\x)$ is a signed measure on the physical space $(0,T)\times \Omega$.
In conclusion, the difference
\begin{eqnarray*}
\mu_G\equiv\widehat{G(\mathbf u)}-[(t,\x)\mapsto \langle Y_{t,\x}; G(\mathbf u)\rangle]\in \mathcal{M}(Q_T),
\end{eqnarray*}
is called \textit{concentration defect measure}.

As it is quite common in works dealing with interactions between fluids and rigid bodies, one can combine the momentum equation for the fluid velocity with the equations for the velocities of the body into one weakly formulated equation. We provide the relevant calculations in Section \ref{App-Euler} and this motivates us to introduce the measure-valued solution to the system \eqref{fluidmotion:ineuler}-\eqref{bodymotion:ineuler} as follows.

\begin{Definition} \label{mvs}
Let $T > 0$, $\Omega$ and $\B_0 \subset \Omega$ be bounded domains and let $Y_{0,\x} \in L^\infty_{weak-*}(\Omega;\mathcal{P}(\R^3))$ such that
\begin{equation}\label{bodyYoung00}
Y_{0,\x} = \delta_{\vV_0 + \vw_0 \times \x} \qquad \text{ for all } \x \in \mc{B}_0
\end{equation}
for some $\vV_0,\vw_0 \in \R^3$.

We say that a triplet $(\B,Y_{t,\x}, \D)$
represents a measure-valued solution for the system \eqref{fluidmotion:ineuler}-\eqref{bodymotion:ineuler} on the set $(0,T)\times\Omega$ with the initial data $\B_0$, $Y_{0,\x}$, if the following holds
\begin{enumerate}
\item $\mc{B}(t) \subset \Omega$ is a bounded domain of $\mathbb{R}^3$ for all $t\in [0,T)$ such that
\begin{equation}\label{body}
\chi_{\mc{B}}(t,\x) = \mathds{1}_{\mc{B}(t)}(\x) \in L^{\infty}((0,T) \times \Omega)
\end{equation}
and there exists a family of isometries $\eta[t]$ of $\R^3$ such that \eqref{eq:isom1}-\eqref{eq:isom2} are satisfied.

\item $Y_{t,\x} \in L^\infty_{weak-*}((0,T)\times\Omega;\mathcal{P}(\R^3))$ such that 
\begin{equation}\label{bodyYoung}
Y_{t,\x} = \delta_{\vu_{\mathcal{B}}(t,\x)} \qquad \text{ for all } (t,\x) \in Q_{\mc{B}},
\end{equation}
where $\vu_{\mc{B}}(t,\x) = \vV(t) + \vw(t)\times (\x-\X(t))$, \eqref{eq:QO}-\eqref{eq:Qw} are satisfied and $\X(t)$ and $\tn{O}(t)$ are absolutely continuous on $[0,T]$. 

\item The dissipation defect 
$\mathcal{D} \in L^\infty(0,T), \ \mathcal{D}\geq0$.

\item The continuity equation is satisfied as follows: for any test function  $\varphi \in C^\infty_c([0,\tau)\times \overline{\Omega})$, $\tau \in [0,T]$
\begin{equation}\label{solenoidality}
 \int_0^\tau\int_{\mc{F}(t)} \langle Y_{t,\x},\vu_{\mc{F}}\rangle \cdot  \nabla \varphi= \int_0^\tau\int_{\partial \mc{B}(t)}   \vu_{\mc{B}}\cdot \vc{n} \, \varphi.
\end{equation}

\item The momentum equation is satisfied as follows: for any test function $\vphi \in V_\tau$, $\tau\in [0,T]$
\begin{multline}\label{momentum:euler}
- \int_0^\tau\int_{\mc{F}(t)} \langle Y_{t,\x},\vu_{\mc{F}}\rangle\cdot {\partial_ t}\vphi_{\mc{F}} - \int_0^\tau\int_{\mc{B}(t)} \vr_\B \vu_{\mc{B}}\cdot {\partial_ t}\vphi_{\mc{B}} - \int_0^\tau\int_{\mc{F}(t)} \langle Y_{t,\x},(\vu_{\mc{F}} \otimes \vu_{\mc{F}}) : \nabla \vphi_{\mc{F}}\rangle
 \\ =  \int_{\mc{F}_0}\langle Y_{0,\x}, \vu_{\mc{F}}\rangle\cdot \vphi_{\mc{F}}(0)
 -\int_{\mc{F}_\tau}\langle Y_{\tau,\x}, \vu_{\mc{F}}\rangle\cdot \vphi_{\mc{F}}(\tau)
 + \int_{\mc{B}_0} (\vr_\B \vu_{\mc{B}}\cdot \vphi_{\mc{B}})(0)
 - \int_{\mc{B}_\tau} (\vr_\B \vu_{\mc{B}}\cdot \vphi_{\mc{B}})(\tau)
 + \int_{0}^\tau \langle \mu^M_{D}, \nabla \vphi_{\mc{F}}\rangle
\end{multline}
with some measure $\mu_D^M \in L^1(0,T;\mathcal{M}(\F_t))$.
\item The energy inequality
\begin{equation}\label{energy:euler}
\int_{\mc{F}(\tau)} \langle Y_{\tau,\x},\frac{1}{2}|\vu_{\mc{F}}|^2\rangle + \int_{\mc{B}(\tau)} \frac{1}{2}\vr_\B |\vu_{\mc{B}}|^2
+\mathcal{D}(\tau) \leq E_{0},
\end{equation}
holds for a.a. $\tau \in [0,T]$ with $E_{0}=\int_{\mc{F}_0} \langle Y_{0,\x},\frac{1}{2}|\vu_{\mc{F}}|^2\rangle + \int_{\mc{B}_0} \frac{1}{2}(\vr_\B |\vu_{\mc{B}}|^2) |_{t=0}$.
\item The following compatibility condition holds: There exists $\xi\in L^1(0,T)$ such that
\begin{equation}\label{cc:euler}
|\langle \mu_{D}^{M}(\tau),\nabla \vphi_\F \rangle| \leq \xi(\tau)\mathcal{D}(\tau)\|\vphi_\F\|_{C^1(\overline{\F_\tau})}
\end{equation}
for a.a. $\tau \in (0,T)$ and every $\vphi \in V_\tau$.
\end{enumerate}
\end{Definition}

\begin{Remark}
Points $1.$ and $2.$ of Definition \ref{mvs} imply in particular that the rigid body $\mc{B}$ is transported by the rigid vector field $\vu_{\mc{B}}$, which can be written in the weak sense as follows: for all $\psi\in C^\infty_c([0,\tau)\times \overline{\Omega})$, $\tau \in [0,T]$ both of the following identities hold
\begin{equation}\label{bodytransport}
-\int_{0}^\tau\int_{\mc{B}(t)} \partial_t \psi - \int_{0}^\tau\int_{\mc{B}(t)}\vu_{\mc{B}}\cdot \nabla \psi = \int_{\mc{B}_0}(\psi) |_{t=0}.
\end{equation}
\begin{equation}\label{bodytransport2}
-\int_{0}^\tau\int_{\mc{B}(t)} \vr_\B \partial_t \psi - \int_{0}^\tau\int_{\mc{B}(t)}\vr_\B \vu_{\mc{B}}\cdot \nabla \psi = \int_{\mc{B}_0}(\vr_\B \psi) |_{t=0}.
\end{equation}
The opposite implication holds as well, namely the fact, that the rigid body $\B$ is transported by a rigid motion $\vu_\B$ such that \eqref{bodytransport} holds implies the existence of a family of isometries $\eta[t]$ with properties \eqref{eq:isom1}-\eqref{eq:isom2} related to the rigid velocity $\vu_\B$ through \eqref{eq:QO}-\eqref{eq:Qw}. For more details we refer to \cite{GH2}.
\end{Remark}

\begin{Remark}
We also note here that for any $\tau \in [0,T]$
\begin{equation}\label{eq:energy_body}
 \frac12 \int_{\B(\tau)} \vr_\B |\vu_{\B}|^2 = \frac m2 |\vV(\tau)|^2 + \frac 12 \mathbb{J}(\tau)\vw(\tau)\cdot\vw(\tau) \geq c\left(|\vV(\tau)|^2 + |\vw(\tau)|^2\right)
\end{equation}
for some constant $c > 0$ which is independent of time.
\end{Remark}
\begin{Remark}
Let us mention that point $2.$ of Definition \ref{mvs} introduces a generalized version of compatibility between the family of isometries $\eta[t]$ and the Young measure $Y_{t,\x}$ in the sense of Feireisl \cite{F3}.
\end{Remark}

In order to establish the existence result \cref{thm:measure-euler} for the system \eqref{fluidmotion:ineuler}-\eqref{bodymotion:ineuler}, we first introduce the  following system by replacing Euler equations by Navier-Stokes equations (with suitable boundary conditions):
\begin{equation}\label{fluidmotion:inNSE}
\left\{
\begin{array}{ccc}
\partial_{t} \vu^{\varepsilon}_{\mathcal{F}} + (\vu^{\varepsilon}_{\mathcal{F%
}} \cdot \nabla) \vu^{\varepsilon}_{\mathcal{F}} - \operatorname{div}\sigma(\vu^{\varepsilon}_{\mc{F}},p^{\varepsilon}_{\mathcal{F}}) = 0,\quad \mathrm{div} \vu^{\varepsilon}_{\mathcal{F}}=0 & \mathrm{in} & {Q}_{%
\mathcal{F}^\varepsilon},\\
\vu^{\varepsilon}_{\mathcal{F}} \cdot \vc{n} = 0 & \mathrm{on} & (0,T) \times \partial\Omega , \\
\vu^{\varepsilon}_{\mathcal{F}} \cdot \vc{n} = \vu^{\varepsilon}_{\mathcal{B}} \cdot \vc{n}
&\mathrm{on} & Q_{\partial
\mathcal{B}^\varepsilon}, \\
\sigma(\vu^\varepsilon_\F,p^\varepsilon_\F)\vc{n} \times \vc{n} = 0 & \mathrm{on} &  (0,T) \times \partial\Omega \\
\sigma(\vu^\varepsilon_\F,p^\varepsilon_\F)\vc{n} \times \vc{n} = 0 & \mathrm{on} & Q_{\partial
\mathcal{B}^\varepsilon} \\
\displaystyle
m\frac{\d}{\dt}\vV^{\varepsilon}(t) = -\int_{\partial\mathcal{B}^\varepsilon(t)} \sigma(\vu^{\varepsilon}_{\mc{F}},p^{\varepsilon}_{\mathcal{F}}) \vc{n}\ \d S\quad&
 \text{ in } &(0,T), \\ \displaystyle
\tn{J}(t)\frac{\d}{\dt}\vw^{\varepsilon}(t) = \tn{J}(t)\vw^{\varepsilon}(t)\times\vw^{\varepsilon}(t) - \int_{\partial\mathcal{B}^\varepsilon(t)} (\x-\X^\varepsilon(t)) \times \sigma(\vu^{\varepsilon}_{\mc{F}},p^{\varepsilon}_{\mathcal{F}}) \vc{n}\ \d S\quad &\text{ in }& (0,T),\\
\vu^{\varepsilon}_{\mathcal{F}}(0)= \vu_{\F_{0}} & \mathrm{in} & \mathcal{F%
}_{0},\\
\vV^{\varepsilon}(0)=\vV_{0},\qquad \vw^{\varepsilon}(0)= \vw_{0}.
\end{array}%
\right.
\end{equation}
In the above, $\vu^{\varepsilon}_{\mathcal{F}}$ denotes the velocity of the fluid, the scalar function $p^{\varepsilon}_{\mc{F}}$ is the pressure, and the velocity of the rigid body is given by
\begin{equation*}
\vu_{\mathcal{B}}^\varepsilon(t,\x)=\vV^\varepsilon(t) + \vw^\varepsilon(t)\times (\x-\X^\varepsilon(t))
\end{equation*}
for any $(t,\x) \in Q_{\B^\varepsilon}$, we also denote $\vu_{\B_0} = \vV_0 + \vw_0 \times \x$. The Cauchy stress tensor $\sigma(\vu^{\varepsilon}_{\mc{F}},p^{\varepsilon}_{\mathcal{F}})$ is given by
\begin{equation*}
\sigma(\vu^{\varepsilon}_{\mc{F}},p^{\varepsilon}_{\mathcal{F}})=2\varepsilon \mathbb{D}(\vu^{\varepsilon}_{\mc{F}}) -p^{\varepsilon}_{\mc{F}}\mathbb{I} \mbox{ with }
\mathbb{D}(\vu^{\varepsilon}_{\mc{F}})=\frac{1}{2}\left(\nabla \vu^{\varepsilon}_{\mc{F}} + {\nabla \vu^{\varepsilon}_{\mc{F}}}^{\top}\right).
\end{equation*}
At first, we want to introduce the weak solution to the system \eqref{fluidmotion:inNSE}.
\begin{Definition}
Let $T> 0$, $\Omega$ and $\mc{B}_0 \subset \Omega$ be two Lipschitz bounded domains of $\mathbb{R}^3$.  A couple $(\mc{B}^{\varepsilon},\vu^{\varepsilon})$ is a weak solution to system \eqref{fluidmotion:inNSE} with initial data $\vu_{\F_0}$, $\vV_0$, $\vw_0$ if the following holds:
\begin{enumerate}
\item $\mc{B}^{\varepsilon}(t) \subset \Omega$ is a bounded domain of $\mathbb{R}^3$ for all $t\in [0,T)$ such that
\begin{equation}\label{body:NS}
\chi^{\varepsilon}_{\mc{B}}(t,x) = \mathds{1}_{\mc{B}^{\varepsilon}(t)}(x) \in L^{\infty}((0,T) \times \Omega).
\end{equation}
\item $\vu^{\varepsilon}=(1-\chi^{\varepsilon}_{\mc{B}})\vu^{\varepsilon}_{\mc{F}} + \chi^{\varepsilon}_{\mc{B}}\vu^{\varepsilon}_{\mc{B}}$ belongs to the following space
\begin{equation*}
U_{T}=
\left\{\!\begin{aligned}
&\vu^{\varepsilon} \in L^{\infty}(0,T; L^2_{\sigma}(\Omega)),\mbox{ there exists }\vu^{\varepsilon}_{\mc{F}}\in L^2 (0,T; H^1_{\sigma}(\Omega)),\, \vu^{\varepsilon}_{\mc{B}}\in L^{\infty}(0,T; \mc{R})\\ &\mbox{ such that }\vu^{\varepsilon}(t,\cdot)=\vu^{\varepsilon}_{\mc{F}}(t,\cdot)\mbox{ on }\mc{F}^{\varepsilon}(t),\quad \vu^{\varepsilon}(t,\cdot)=\vu^{\varepsilon}_{\mc{B}}(t,\cdot)\mbox{ on }\mc{B}^{\varepsilon}(t)\mbox{ with }\\ &\vu^{\varepsilon}_{\mc{F}}(t,\cdot)\cdot \vc{n} = \vu^{\varepsilon}_{\mc{B}}(t,\cdot)\cdot \vc{n} \mbox{ on }\partial\mc{B}^{\varepsilon}(t),\ \vu^{\varepsilon}_{\mc{F}}\cdot \vc{n}=0 \mbox{ on }\partial\Omega\mbox{ for all }t\in [0,T]
\end{aligned}\right\}.
\end{equation*}
\item The transport of $\mc{B}^{\varepsilon}$ by the rigid vector field $\vu_{\mc{B}} ^{\varepsilon}$:
\begin{equation}\label{bodytransport:NS}
\frac{\partial \chi_{\B}^{\varepsilon}}{\partial t} + \operatorname{div}(\vu_\B^{\varepsilon} \chi_{\B}^{\varepsilon}) = 0 \quad \mbox{ in }\R^3 \mbox{ with }  \chi^{\varepsilon}_{\mc{B}}|_{t=0}= \mathds{1}_{\mc{B}_0}.
\end{equation}
is satisfied in the weak sense: for all $\psi\in C^\infty_c([0,\tau)\times \overline{\Omega})$, $\tau \in [0,T]$
\begin{equation}
-\int_{0}^\tau\int_{\mc{B}^{\varepsilon}(t)} 
\partial_t \psi - \int_{0}^\tau\int_{\mc{B}^{\varepsilon}(t)} 
\vu^{\varepsilon}_{\mc{B}}\cdot \nabla \psi = \int_{\mc{B}_0} 
\psi |_{t=0}.
\end{equation}
\item Balance of linear momentum holds in a weak sense, i.e, for all  $\vphi \in V_{T}$ and $\tau \in [0,T]$ 

\begin{multline}\label{momentum:NS}
- \int_0^\tau\int_{\mc{F}^{\varepsilon}(t)} \vu_{\mc{F}}^{\varepsilon}\cdot \partial_t\vphi_{\mc{F}} - \int_0^\tau\int_{\mc{B}^{\varepsilon}(t)} \vr_\B^\varepsilon \vu_{\mc{B}}^{\varepsilon}\cdot \partial_t\vphi_{\mc{B}} - \int_0^\tau\int_{\mc{F}^{\varepsilon}(t)} (\vu_{\mc{F}}^{\varepsilon} \otimes \vu_{\mc{F}}^{\varepsilon}) : \nabla \vphi_{\mc{F}} + 2\varepsilon\int_0^\tau\int_{\mc{F}^{\varepsilon}(t)} \mathbb{D}(\vu_{\mc{F}}^{\varepsilon}):\mathbb{D}(\vphi_{\mc{F}})\\
  = \int_{\mc{F}_0} \vu_{\mc{F}_0}\cdot \vphi_{\mc{F}}(0) 
  -\int_{\mc{F}^\varepsilon(\tau)} \vu_{\mc{F}}^\varepsilon\cdot \vphi_{\mc{F}}(\tau)
  + \int_{\mc{B}_0} \vr_\B \vu_{\mc{B}_0}\cdot \vphi_{\mc{B}}(0)
  - \int_{\mc{B}^\varepsilon(\tau)} \vr_\B^\varepsilon \vu_{\mc{B}}^\varepsilon\cdot \vphi_{\mc{B}}(\tau).
\end{multline}
\item The energy inequality holds for a.e $\tau\in [0,T]$:
\begin{equation}\label{energy:NS}
\int_{\mc{F}^{\varepsilon}(t)}\frac{1}{2} |\vu_{\mc{F}}^{\varepsilon}(\tau)|^2 + \int_{\mc{B}^{\varepsilon}(t)} \frac{1}{2} \vr_\B^\varepsilon|\vu_{\mc{B}}^{\varepsilon}(\tau)|^2 + 2\varepsilon\int_0^\tau\int_{\mc{F}^{\varepsilon}(t)} |\mathbb{D}(\vu_{\mc{F}}^{\varepsilon})|^2 \leq \frac12 \int_{\mc{F}_0} |\vu_{\mc{F}_0}|^2
   + \frac12 \int_{\mc{B}_0} \vr_\B|\vu_{\mc{B}_0}|^2.
\end{equation}
\end{enumerate}
  \end{Definition}
\begin{Remark}
The density of the body $\vr_\B^\varepsilon$ in \eqref{momentum:NS} is naturally solution of 
\begin{equation}\label{eq:densityNS}
    \partial_t \vr_\B^\varepsilon + \Div (\vr_\B^\varepsilon \vu_\B^\varepsilon) = 0 \qquad \mbox{ in } Q_{\B^\varepsilon}
\end{equation}
with initial data given by $\vr_\B(0,\x)$.
\end{Remark}

We already know the following existence theorem due to \cite[Theorem 2.2]{GH2}:
\begin{Theorem}\label{t:GVH}
Let $\Omega$ and $\mc{B}_0 \subset \subset \Omega$ be two smooth domains of $\R^3$. Let $\vu_{\mc{F}_0} \in L^2_{\sigma}(\Omega)$ and $\vu_{\mc{B}_0} \in \mc{R}$ be such that $\vu_{\mc{F}_0}\cdot \vc{n}=\vu_{\mc{B}_0} \cdot \vc{n}$ on $\partial\mc{B}_0$. There exists $T\in \R^*_{+} \cup \infty$ and a weak solution $(\mc{B}^{\varepsilon},\vu^{\varepsilon})$ to \eqref{fluidmotion:inNSE} on $[0,T)$. Moreover, such a weak solution exists up to collision, i.e, either we can take $T=\infty$ or we can take $T>0$ such that
\begin{equation*}
    \mc{B}^\varepsilon(t) \subset \subset \Omega \mbox{ for all }t\in [0,T),\quad \lim_{t\rightarrow T}\mbox{dist}(\mc{B}^\varepsilon(t), \partial\Omega)=0.
\end{equation*}
\end{Theorem}
\begin{Remark}
In fact, Theorem \ref{t:GVH} is proved in \cite{GH2} for constant $\vr_\B$ and Navier slip boundary conditions, however generalization for non-constant smooth density of the body and complete slip boundary conditions is straightforward.
\end{Remark}

Now we are in a position to prove that the sequence of weak solutions $(\mc{B}^{\varepsilon},\vu^{\varepsilon})$ to system \eqref{fluidmotion:inNSE} converges to a measure-valued solution $(\B,Y_{t,\x},\D)$ for the system \eqref{fluidmotion:ineuler}-\eqref{bodymotion:ineuler} as $\varepsilon \rightarrow 0$. To do that, we recall couple of results.



We use a special case of \cite[Proposition 3.4]{GH2}:
  \begin{Proposition}\label{conv:body}
  Let $(\chi_\mc{B}^{\varepsilon},\vu_\B^{\varepsilon})$ be a bounded sequence in $L^{\infty}((0,T)\times \R^3) \times L^{\infty}(0,T;\mathcal{R})$, such that $\chi_\B^\varepsilon = \mathds{1}_{\mc{B}^\varepsilon}$ satisfying
  \begin{equation*}
  \frac{\partial \chi_{\B}^{\varepsilon}}{\partial t} + \operatorname{div}(\vu_\B^{\varepsilon} \chi_{\B}^{\varepsilon}) = 0 \quad \mbox{ in }\R^3 \mbox{ with }  \chi^{\varepsilon}_{\mc{B}}|_{t=0}= \mathds{1}_{\mc{B}_0}.
  \end{equation*}
  Then up to a subsequence 
\begin{equation*}
\vu_\B^{\varepsilon} \rightarrow \vu_\B \quad \mbox{ weakly-}* \mbox{ in }L^{\infty}(0,T;C^{\infty}_{loc}(\R^3))
\end{equation*}
with $\vu_\B \in L^\infty(0,T;\mathcal{R})$ and 
\begin{equation*}
\chi_{\B}^{\varepsilon} \rightarrow \chi_{\B} \quad \mbox{ weakly-}* \mbox{ in } L^{\infty}((0,T) \times \R^3) \mbox{ and strongly in } C([0,T]; L^p_{loc}(\R^3))    
\end{equation*}
 for $p<\infty$ with $\chi_{\B}$ is a solution to
  \begin{equation*}
  \frac{\partial \chi_{\B}}{\partial t} + \operatorname{div}(\vu_\B \chi_{\B}) = 0 \quad \mbox{ in }\R^3 \mbox{ with }  \chi_{\mc{B}}|_{t=0}= \mathds{1}_{\mc{B}_0}.
  \end{equation*}
  Moreover, $\chi_{\mc{B}}(t,\cdot)=\mathds{1}_{\mc{B}(t)}$ for all $t$.
\end{Proposition}


Secondly, we recall the following comparison result \cite[Lemma 2.1]{FGGE}:
 \begin{Lemma}\label{comparison}
 Let $\{Z_n\}_{n=1}^{\infty}$, $Z_n:\Omega \rightarrow \R^N$ be a sequence of equi-integrable functions generating a Young measure $Y_z$, $z\in \Omega$ with $\Omega \subset \R^N$ is a bounded domain. Let
 \begin{equation*}
 G: \R^N \rightarrow [0,\infty)
 \end{equation*}
 be a continuous function such that
 \begin{equation*}
 \sup\|G(Z_n)\|_{L^1(\Omega)} < \infty
 \end{equation*}
 and let $F$ be continuous such that
 \begin{equation*}
 F: \R^N \rightarrow \R, \quad |F(Z)| \leq G(Z), \quad \forall \ Z\in \R^N.
 \end{equation*}
 Define
 \begin{equation*}
 F_{\infty}=\widetilde{F}-\langle Y_z, F(Z)\rangle\ dz, \quad G_{\infty}=\widetilde{G}-\langle Y_z, G(Z)\rangle\ dz,
 \end{equation*}
 where $\widetilde{F} \in \mc{M}(\overline{\Omega})$, $\widetilde{G} \in \mc{M}(\overline{\Omega})$ are the weak-$*$ limits of $\{F(Z_n)\}_{n=1}^{\infty}$, $\{G(Z_n)\}_{n=1}^{\infty}$ in $\mc{M}(\overline{\Omega})$. Then
 \begin{equation*}
 |F_{\infty}| \leq |G_{\infty}|.
 \end{equation*}
 \end{Lemma}
  \begin{proof}[Proof of \cref{thm:measure-euler}]
  For every $\varepsilon >0$, $\vu_{\mc{F}}^\varepsilon(0)=\vu_{\mc{F}_0} \in L^2_{\sigma}(\Omega)$ and $\vu_{\mc{B}}^\varepsilon(0)=\vu_{\mc{B}_0} \in \mc{R}$ are such that $\vu_{\mc{F}_0}\cdot \vc{n}=\vu_{\mc{B}_0} \cdot \vc{n}$ on $\partial\mc{B}_0$. Let $(\mc{B}^{\varepsilon},\vu^{\varepsilon})$ be a sequence of weak solutions to system \eqref{fluidmotion:inNSE}. We have $\chi^{\varepsilon}_{\mc{B}}(t,\x) = \mathds{1}_{\mc{B}^{\varepsilon}(t)}(\x) \in L^{\infty}((0,T) \times \Omega)$ from \eqref{body:NS} and $\vu^{\varepsilon}$ is a bounded sequence in $L^{\infty}(0,T;L^2_{\sigma}(\Omega))$ from energy estimate \eqref{energy:NS} with $\vu_\B^\varepsilon \in L^\infty(0,T;\mathcal{R})$. Thus, we can use \cref{conv:body} to establish the relations \eqref{body} and \eqref{bodytransport} from the equations  \eqref{body:NS} and \eqref{bodytransport:NS} as $\varepsilon \rightarrow 0$. By the fundamental theorem of Young measures we denote $Y_{t,\x}$ the  parametrized probability measure generated by the sequence $\vu^\varepsilon$. Again by Proposition \ref{conv:body} we moreover get \eqref{bodyYoung}. Therefore we denote by $\vu_\F$ the dummy variable for the Young measure $Y_{t,\x}$ on the fluid part, whereas we denote $\vu_\B$ directly the limit of $\vu^\varepsilon$ on the solid part.

The divergence-free condition $\vu^\varepsilon \in L^\infty(0,T,L^2_\sigma(\Omega))$ together with boundary conditions on $\partial\Omega$ and $\partial\B^\varepsilon$ imply in particular that for any test function  $\varphi \in C^\infty_c([0,\tau)\times \overline{\Omega})$, $\tau \in [0,T]$
\begin{equation}\label{solenoidality0}
\int_0^\tau\int_{\mc{F}^\varepsilon(t)} \vu^\varepsilon_{\mc{F}} \cdot  \nabla \varphi= \int_0^\tau\int_{\partial \mc{B}^\varepsilon(t)}   \vu_{\mc{B}}^\varepsilon\cdot \vc{n} \, \varphi.
\end{equation}
Observe that
\begin{equation*}
\int_0^\tau\int_{\mc{F}^\varepsilon(t)} \vu^\varepsilon_{\mc{F}} \cdot  \nabla \varphi=\int_0^\tau\int_{\Omega} (1-\chi^{\varepsilon}_{\mc{B}})\vu^\varepsilon\cdot  \nabla \varphi, \quad  
\int_0^\tau\int_{\partial \mc{B}^\varepsilon(t)}   \vu_{\mc{B}}^\varepsilon\cdot \vc{n} \, \varphi= 
\int_0^\tau\int_{\mc{B}^\varepsilon(t)} \vu^\varepsilon_{\mc{B}} \cdot  \nabla \varphi=
\int_0^\tau\int_{\Omega} \chi^{\varepsilon}_{\mc{B}}\vu^\varepsilon\cdot  \nabla \varphi.
\end{equation*}
\cref{conv:body} helps us in passing the limit in \eqref{solenoidality0} as $\varepsilon\rightarrow 0$ to obtain \eqref{solenoidality}.
Now we can concentrate on the momentum equation. Notice that
 \begin{equation*}
 \varepsilon\int_0^T\int_{\mc{F}^{\varepsilon}(t)} \mathbb{D}(\vu_{\mc{F}}^{\varepsilon}):\mathbb{D}(\vphi_{\mc{F}}) = \sqrt{\varepsilon} \int_0^T\int_{\mc{F}^{\varepsilon}(t)} \sqrt{\varepsilon}\mathbb{D}(\vu_{\mc{F}}^{\varepsilon}):\mathbb{D}(\vphi_{\mc{F}}) \leq \sqrt{\varepsilon} \|\sqrt{\varepsilon}\mathbb{D}(\vu_{\mc{F}}^{\varepsilon})\|_{L^2(0,T; L^2(\mc{F}^{\varepsilon}))}\|\mathbb{D}(\vphi_{\mc{F}})\|_{L^2(0,T; L^2(\mc{F}^{\varepsilon}))}.
 \end{equation*}
The energy estimate \eqref{energy:NS} implies $\|\sqrt{\varepsilon}\mathbb{D}(\vu_{\mc{F}}^{\varepsilon})\|_{L^2(0,T; L^2(\mc{F}^{\varepsilon}))} \leq C$. Thus the term
\begin{equation*}
\varepsilon\int_0^T\int_{\mc{F}^{\varepsilon}(t)} \mathbb{D}(\vu_{\mc{F}}^{\varepsilon}):\mathbb{D}(\vphi_{\mc{F}}) \rightarrow 0 \mbox{ as }\varepsilon \rightarrow 0.
\end{equation*}
 Now we can analyse the first two terms of left-hand side of
\eqref{momentum:NS}. We can write
\begin{equation*}
\int_0^T\int_{\mc{F}^{\varepsilon}(t)} \vu_{\mc{F}}^{\varepsilon}\cdot \frac{\partial}{\partial t}\vphi_{\mc{F}} + \int_0^T\int_{\mc{B}^{\varepsilon}(t)} \vr_\B^\varepsilon \vu_{\mc{B}}^{\varepsilon}\cdot \frac{\partial}{\partial t}\vphi_{\mc{B}} = \int_0^T\int_{\Omega}\left[ (1-\chi_{\B}^{\varepsilon})\vu^{\varepsilon}\cdot \frac{\partial}{\partial t}\vphi_{\mc{F}} + \vr_\B^\varepsilon\chi_{\B}^{\varepsilon}\vu^{\varepsilon}\cdot \frac{\partial}{\partial t}\vphi_{\mc{B}}\right]
\end{equation*}
with $\vu^{\varepsilon}=(1-\chi^{\varepsilon}_{\mc{B}})\vu^{\varepsilon}_{\mc{F}} + \chi^{\varepsilon}_{\mc{B}}\vu^{\varepsilon}_{\mc{B}}$. 
We already know from \cref{conv:body} that $\chi_{\B}^{\varepsilon} \rightarrow \chi_{\B}$ strongly in $C([0,T]; L^p_{loc}(\R^3))$ for $p<\infty$, therefore also $\vr_\B^\varepsilon\chi_{\B}^{\varepsilon} \rightarrow \vr_\B\chi_{\B}$ strongly in $C([0,T]; L^p_{loc}(\R^3))$ for $p<\infty$. From the choice of the test function space $V_T$, it is clear that $\partial_t \vphi_{\mc{F}}\in \mc{D}([0,T); \mc{D}(\Omega))$, $\partial_t\vphi_{\mc{B}}\in \mc{D}([0,T); \mc{R})$. Using the above mentioned tools, we have
\begin{equation*}
\int_0^T\int_{\Omega}\left[ (1-\chi_{\B}^{\varepsilon})\vu^{\varepsilon}\cdot \frac{\partial}{\partial t}\vphi_{\mc{F}} + \vr_\B^\varepsilon \chi_{\B}^{\varepsilon}\vu^{\varepsilon}\cdot \frac{\partial}{\partial t}\vphi_{\mc{B}}\right] \rightarrow \int_0^T\int_{\Omega} \left[(1-\chi_{\B})\langle Y_{t,\x},\vu_{\mc{F}}\rangle\cdot \frac{\partial}{\partial t}\vphi_{\mc{F}} + \vr_\B \chi_{\B}\vu_{\B}\cdot \frac{\partial}{\partial t}\vphi_{\mc{B}}\right].
\end{equation*}
Now we analyze the convergence of the convective term in the momentum equation and the energy inequality. Regarding the diffusive term in the energy inequality \eqref{energy:NS}:
\begin{equation*}
2\varepsilon\int_0^T\int_{\mc{F}^{\varepsilon}(t)} |\mathbb{D}(\vu_{\mc{F}}^{\varepsilon})|^2 \geq 0 \mbox{ as }\varepsilon \rightarrow 0.
\end{equation*}

 We have from the energy inequality \eqref{energy:NS} that $((1-\chi_{\mc{B}}^{\varepsilon})+\chi_{\mc{B}}^{\varepsilon}\vr_\B^\varepsilon)\frac{1}{2} |\vu^{\varepsilon}|^2$ is bounded in $L^{\infty}(0,T;L^1(\Omega))$. We can identify
 \begin{equation*}
 ((1-\chi_{\mc{B}}^{\varepsilon})+\chi_{\mc{B}}^{\varepsilon}\vr_\B^\varepsilon)\frac{1}{2} |\vu^{\varepsilon}|^2 (\tau,\cdot) \in \mc{M}(\overline{\Omega}) \mbox{ bounded uniformly for }\tau\in [0,T].
 \end{equation*}
 Up to a subsequence, we can assume
 \begin{equation*}
 ((1-\chi_{\mc{B}}^{\varepsilon})+\chi_{\mc{B}}^{\varepsilon}\vr_\B^\varepsilon)\frac{1}{2} |\vu^{\varepsilon}|^2 (\tau,\cdot) \rightarrow E \mbox{ weakly-}* \mbox{ in }L^{\infty}_{weak}(0,T; \mc{M}(\overline{\Omega})).
 \end{equation*}
 We introduce a new non-negative measure:
 \begin{equation*}
 E_{\infty}= E- \langle Y_{t,\x}; ((1-\chi_{\mc{B}}^{\varepsilon})+\chi_{\mc{B}}^{\varepsilon}\vr_\B^\varepsilon) \frac{1}{2}|\vu^{\varepsilon}|^2\rangle\ d \x.
 \end{equation*}
 We can take the limit $\varepsilon \rightarrow 0$ in the energy balance \eqref{energy:NS} that yield: for a.e. $\tau\in (0,T)$
 \begin{equation}\label{energy:euler_2}
\int_{\mc{F}(\tau)} \langle Y_{\tau,\x},\frac{1}{2}|\vu_{\mc{F}}|^2\rangle + \int_{\mc{B}(\tau)} \frac{1}{2}\vr_\B |\vu_{\mc{B}}|^2
+E_{\infty}(\tau)|\Omega| \leq \int_{\mc{F}_0} \frac{1}{2}|\vu_{\mc{F}_0}|^2 + \int_{\mc{B}_0} \frac{1}{2}\vr_\B |\vu_{\mc{B}_0}|^2.
\end{equation}
 Thus, it gives us the energy inequality \eqref{energy:euler} with
 \begin{equation*}
 \mc{D}(\tau)=E_{\infty}(\tau)|\Omega| \mbox{ for a.e. }\tau\in (0,T).
 \end{equation*}
 Regarding the convective term, the quantity $\vu^{\varepsilon}_\F \otimes \vu^{\varepsilon}_\F$ is bounded only in $L^{\infty}(0,T;L^1(\mc{F}^{\varepsilon}))$. Moreover, we have the following observation:
 \begin{equation*}
 |\vu^\varepsilon_\F \otimes \vu^\varepsilon_\F| \leq |\vu^\varepsilon_\F|^2.
 \end{equation*}
 Thus, using \cref{comparison}, we have the following convergence: as $\varepsilon \rightarrow 0$
 \begin{equation*}
\int_0^T\int_{\mc{F}^{\varepsilon}(t)} (\vu_{\mc{F}}^{\varepsilon} \otimes \vu_{\mc{F}}^{\varepsilon}) : \nabla \vphi_{\mc{F}} \rightarrow \int_0^\tau\int_{\mc{F}(t)} \langle Y_{t,\x},(\vu_{\mc{F}} \otimes \vu_{\mc{F}}) : \nabla \vphi_{\mc{F}}\rangle + \int_{0}^\tau \langle \mu^M_{D}, \nabla \vphi_\F\rangle,
 \end{equation*}
 where
 \begin{equation*}
 \mu^M_{D}=\{\mu^M_{D,i,j}\}_{i,j=1}^3,\ \mu^M_{D,i,j}\in L^{\infty}_{weak}(0,T;\mc{M}(\F_t)),
 \end{equation*}
 along with the relation
 \begin{equation*}
 \int^\tau_0\int_{\Omega}|\mu_{D}^{M}
|\leq 2\int^\tau_0 E_{\infty}(\tau)|\Omega| \mbox{ for a.e. }\tau\in (0,T).
\end{equation*}
Thus, we have established the relations \eqref{momentum:euler} and \eqref{energy:euler} with $\xi=2$ and $\mc{D}(\tau)=E_{\infty}(\tau)|\Omega| \mbox{ for a.e. }\tau\in (0,T)$.
\end{proof}

\section{Weak-strong uniqueness}\label{s:wsu}

\subsection{Change of coordinates}\label{s:cc}

Let $\{{\B_1},Y_{t,\x},\D\}$ be a measure-valued solution to \eqref{fluidmotion:ineuler}-\eqref{bodymotion:ineuler} on time interval $(0,T)$ in the sense of Definition \ref{mvs} with
$\vu_{1\B} = \vV_1 + \vw_1 \times (\x - \X_1)$ being the associated velocity of the rigid body. Moreover we denote by $\vu_{1 \F}$ the dummy variable of the Young measure $Y_{t,\x}$ on the fluid domain.
Let $(\B_2, \vu_{2 \F}, p_{2 \F}, \vV_2,\vw_2)$ be a strong solution to \eqref{fluidmotion:ineuler}-\eqref{bodymotion:ineuler} on time interval $(0,T_0)$. We denote by $\F_1(t)$ the domain occupied by the fluid at time $t$ for the measure-valued solution and analogously, $\F_2(t)$ for the strong solution.

We denote
\begin{equation} \label{u2}
\vu_2 =
\left|
\begin{array}{c}
\vu_{2 \mathcal{F}}\ \ \mbox{in}\ Q_{\F_2},\\\
\vV_2 + \vw_2\times(\x-\X_2)\ \ \mbox{in}\ Q_{\B_2}.
\end{array}
\right.
\end{equation}
A priori, there is no reason why the body $\B_1(t)$ of the measure-valued solution would not touch the boundary $\de\Omega$ at some time instant $t < T_0$. Therefore, following Kreml et al. \cite{KrNePi_2}, we introduce a time $T_{min}$ such that
\begin{equation} \label{eq:tmin}
    T_{min} = \inf \left\{t \in (0,T_0); d(\B_1(t),\de\Omega) \leq \frac{\kappa}{2}\right\},
\end{equation}
where, for the strong solution, $\kappa$ satisfies the following relation
\begin{equation}\label{eq:dist_cond}
d(\B_2(t),\de \Omega)\geq \kappa>0, \quad \forall t\in[0,T_0], \quad \textrm{for some} \quad d(\B_0,\de \Omega)>\kappa>0.
\end{equation}
In particular, we have $T_{min} > 0$ and on the interval $(0,T_{min})$ there is no collision between the body $\B_1$ and the boundary $\de \Omega$. Here and hereafter, our analysis will be performed on the time interval $(0,T_{min})$. However, for simplicity of the notation, we keep denoting the time interval by $(0,T)$.

In order to compare the two solutions, we need to transfer them to the same domain, because in general $\B_1 \neq \B_2$. In particular, we need to transform the strong solution to the domain of the measure-valued solution in such a way that $\B_2(t)$ is transformed to $\B_1(t)$.
For this purpose, following the analysis developed in \cite{KrNePi_2}, we introduce cutoff functions $\zeta_i(t,\x)$ for $i=1,2$ such that $\zeta_i(t,\x)=1$ in a neighborhood of $\B_i(t)$ and $\zeta_i(t,\x)=0$ in a neighborhood of $\de \Omega$. Moreover, we extend the domain of definition of the rigid body motions $\vu_{i\B}$ to the whole $(0,T)\times \Omega$, namely for $\x \in \Omega$ we have
\begin{equation*}
\vu_{i \mathcal{B}}(t,\x)=\vV_i(t) + \vw_i(t)\times (\x-\X_i(t)).
\end{equation*}
Then we set
\begin{equation*}
\Lambda_i(t,\x)=\zeta_i(t,\x)\vu_{i\B}(t,\x).
\end{equation*}
Note that $\Lambda_i$ is a rigid motion in a neighborhood of $\B_i(t)$ and vanishes in a vicinity of $\de \Omega$. In addition, $\Lambda_i$ should be smooth in the space variables and divergence free, in order to preserve the divergence-free condition on the fluid velocity.

Now, we introduce the transformations
$$
\bZ_i:\Om \to \Om
$$
as solutions to the following ODEs
\begin{align*}
&\frac{\d}{\d t}\bZ_i(t,\y)=\Lambda_i(t,\bZ_i(t,\y)), \quad \forall \y \in \Om, \; t \in (0,T),\\
&\bZ_i(0,\y)=\y.
\end{align*}
Then, we define $\bY_i=\bZ_i^{-1}$ and the mappings $\widetilde{\bZ}_i:\Om \to \Om$ as
\begin{align}
\widetilde{\bZ}_1(t,\x)&=\bZ_1(t,\bY_2(t,\x)), \\
\widetilde{\bZ}_2(t,\x)&=\bZ_2(t,\bY_1(t,\x)).
\end{align}
Note that $\widetilde{\bZ}_1(t,\cdot) = \widetilde{\bZ}_2^{-1}(t,\cdot)$ and $\widetilde{\bZ}_2(t,\B_1(t)) = \B_2(t)$ for all $t \in [0,T)$. In particular in the neighborhoods of the body $\B_{2}(t)$ the mapping $\widetilde{\bZ}_1(t,\cdot)$ is rigid and vice versa. More precisely we have
\begin{align}
\widetilde{\bZ}_1(t,\x) &= \X_1(t) + \tn{O}_1(t)\tn{O}_2^T(t)(\x - \X_2(t)) \qquad \text{ in the neighborhood of } \B_2(t), \\
\widetilde{\bZ}_2(t,\x) &= \X_2(t) + \tn{O}_2(t)\tn{O}_1^T(t)(\x - \X_1(t)) \qquad \text{ in the neighborhood of } \B_1(t). \label{eq:bzexpress}
\end{align}
For simplicity of notation we denote
\begin{equation} \label{def:tildeo}
\Ot(t) = \tn{O}_2(t)\tn{O}_1^T(t).
\end{equation}
Now, we define the transformed strong solution  $\bU^s_{\mathcal{F}}$ as
\begin{align}
\label{newsol2}
\bU^s_{\mathcal{F}}(t,\x) &=\tn{J}_{\widetilde{\bZ}_1}(t,\widetilde{\bZ}_2(t,\x))\vu_{2\mathcal{F}}(t,\widetilde{\bZ}_2(t,\x)) \quad \text{ for all } (t,\x) \in Q_T,
\end{align}
where $(\tn{J}_{\widetilde{\bZ}_1})_{ij}(t,\widetilde{\bZ}_2(t,\x)) = \frac{\de (\widetilde{\bZ}_1)_i}{\de \x_j}(t,\widetilde{\bZ}_2(t,\x))$, and the transformed pressure as $P^s_\F(t,\x) =p_{2 \F}(t,\widetilde{\bZ}_2(t,\x))$.
In particular, for $\x \in \B_1(t)$, the transformed rigid velocity reads
\begin{equation} \label{newsol3}
    \bU^s_\B(t,\x) = \vV^s(t) + \vw^s(t) \times (\x-\X_1(t))
\end{equation}
with
\begin{equation}\label{eq:strongbody}
    \vV^s(t) = \Ot^T(t)\vV_2(t), \qquad \qquad \vw^s(t) = \Ot^T(t)\vw_2(t).
\end{equation}
The following lemma holds (see \cite{KrNePi_2}, Lemma 3.1).
\begin{Lemma}\label{l:31}
\begin{itemize}
    \item[(i)] It holds
    \begin{equation}
        \Ot^T(t)\frac{\d \Ot}{\dt}(t)\x = (\vw^s-\vw_1)(t)\times \x.
    \end{equation}
    \item[(ii)] The following estimates hold for $\x \in \de \B_1(t)$, $t \in (0,T)$
    \begin{align}\label{eq:315}
        |\widetilde{\bZ}_2(t,\x) - \x| &\leq C\left(\|\vV_1-\vV^s\|_{L^2(0,t)} + \|\vw_1-\vw^s\|_{L^2(0,t)}\right),\\ \label{eq:316}
        |\de_t \widetilde{\bZ}_2(t,\x)| &\leq C\left(|\vV_1-\vV^s|(t) + |\vw_1-\vw^s|(t)\right).
    \end{align}
    \item[(iii)] The following estimates hold for $t \in (0,T)$
    \begin{align}\label{eq:317}
        \|\widetilde{\bZ}_2(t,\cdot) - \mathrm{id}\|_{W^{3,\infty}(\F_1(t))} &\leq C\left(\|\vV_1-\vV^s\|_{L^2(0,t)} + \|\vw_1-\vw^s\|_{L^2(0,t)}\right),\\
        \|\de_t \widetilde{\bZ}_2(t,\cdot)\|_{W^{1,\infty}(\F_1(t))} &\leq C\left(|\vV_1-\vV^s|(t) + |\vw_1-\vw^s|(t)\right). \label{eq:318}
    \end{align}
\end{itemize}
\end{Lemma}
We would like to mention that such transformation was introduced by Inoue, Wakimoto \cite{inoue1977existence} and properties of it were discussed in details by Takahashi, see \cite{T}. The same transformation has been used in \cite{CNM,KrNePi_2,MNR} to show the weak-strong uniqueness property for the incompressible and compressible fluid-structure interaction problem.

\subsection{Transformed solution and weak formulation}
The transformed strong solution $\vU^s$ satisfies pointwise the following system of equations in the fluid part of the domain $\Omega$
\begin{equation}\label{eq:fluidtransformed}
\left\{
\begin{array}{ccc}
\de_t\bU^s_{\F} + \dv(\bU^s_{\F} \otimes \bU^s_{\F})
+\nabla P^s_{\F} = \mathbf{F}(\bU^s_{\F}),\quad \dv \bU^s_{\F}=0 & {\rm in} & Q_{\F_1}, \\
\bU^s_{\mathcal{F}} \cdot \vc{n} = \bU^s_{\mathcal{B}} \cdot \vc{n} & \mathrm{on} & (0,T) \times \partial\mc{B}_1 , \\
\displaystyle
\bU^s_{\mathcal{F}} \cdot \vc{n}=0 & {\rm on} & (0,T) \times \de \Omega,\\
\displaystyle
\bU^s_{\F}(0,\cdot)= \vu_{\F_0} & {\rm in} &  \F_0.
\end{array}
\right.
\end{equation}
In order to express the term on the right hand sides of the momentum equation, we introduce the following notation
\begin{align} \label{eq:Hdef}
    \tn{H} &= (\nabla_x\ZZ)^{-1} = \nabla_x \ZZj, \quad \text{ i.e. } (\tn{H})_{ij} = \de_{j} (\ZZj)_i,\\ \label{eq:GGdef}
    \tn{G} &= \tn{H}\tn{H}^T, \quad \text{ i.e. } (\tn{G})_{ij} = \de_{k} (\ZZj)_i \de_{k} (\ZZj)_j, \\
    \Gamma^i_{\alpha \beta} &= \de_{l} (\ZZj)_i \de_{\alpha\beta} (\ZZ)_l. \label{eq:Gammadef}
\end{align}
Consequently, we have
\begin{align}
\label{eq:Fdef}
\mathbf{F}_i(\vU^s_{\F}) &= -\tn{H}_{i\alpha}\de_t\de_\beta(\ZZ)_\alpha (\vU^s_{\F})_\beta
+ \Gamma^i_{\alpha\beta}
(\de_t \ZZj)_\alpha (\vU^s_{\F})_\beta
+
(\de_t \ZZj)_\beta \de_\beta (\vU^s_{\F,i})
-\Gamma^i_{\alpha\beta} (\vU^s_{\F})_\alpha(\vU^s_{\F})_\beta
- (\tn{G}-\tn{I})_{i\beta}\de_\beta P^s_{\F}.
\end{align}

The velocity of the rigid body is given by \eqref{newsol3}, where the transformed translation velocity $\vV^s$ and the transformed angular velocity $\vw^s$ of the body satisfy the following system of the equations
\begin{align} \label{eq:newrigidtransformed1}
    m\frac{\d \vV^s}{\dt} &= -m (\vw^s-\vw_1)\times \vV^s + \int_{\partial\mathcal{B}_1(t)}
    P^s_{\F}\tn{I}\vc{n}\, \d S\quad
 \text{ in } (0,T), \\ \label{eq:newrigidtransformed2}
 \tn{J}_1\frac{\d \vw^s}{\dt} &= \tn{J}_1\vw^s\times\vw^s - \left((\vw^s-\vw_1)\times \tn{J}_1 \vw^s\right) + \int_{\partial\mathcal{B}_1(t)} (\x-\X_1(t)) \times P^s_{\F}\tn{I}\vc{n}\, \d S\quad \text{ in } (0,T),\\ \label{eq:newrigidtransformed3}
 \vV^s(0) &= \vV_{0},\qquad \vw^s(0)= \vw_{0},
\end{align}
where we have used $\tn{J}_1 = \Ot^T \tn{J}_2\Ot$.

Relations \eqref{eq:Fdef}-\eqref{eq:newrigidtransformed3} are obtained from a similar computation of the proof of \cite[Lemma 3.1]{CNM}.
A consequence of the Lemma \ref{l:31} is the following lemma that gives a control on the terms of \eqref{eq:Fdef} (see \cite[Lemma 3.3]{CNM}).
\begin{Lemma}\label{l:32}
The following estimate holds
\begin{eqnarray}\label{est:F}
\Vert \mathbf{F} \Vert
_{L^{2}(0,T;L^{2}(\mc{F}_1(t)))}
\leq C\left( ||\mathbf{V}_{1}-\mathbf{V}^{s}||_{L^{2}(0,T)}+||%
\vw_1-\vw^s||_{L^{2}(0,T)} \right),
\end{eqnarray}%
where {$C$ depends only on $\Vert \mathbf{U}^{s}_{\F}\Vert
_{L^{2}(0,T;H^{2}(\mc{F}_1(t)))}$, $\Vert P^s_{\F}\Vert
_{L^{2}(0,T;H^{1}(\mc{F}_1(t)))}$ and $\Vert \mathbf{U}^{s}\Vert
_{L^{\infty }(0,T;H^{1}(\mc{F}_1(t)))}$.}\
\end{Lemma}
Given $\mathbf{F}$ as in \eqref{eq:Fdef}, we can write the weak formulation for the strong solution.
Using \eqref{eq:fluidtransformed}-\eqref{eq:newrigidtransformed3} we obtain that the transformed solution satisfies the following equality for $\tau \in [0,T]$. We present more details about derivation of \eqref{strong} in Section \ref{ss:43}.
\begin{multline}\label{strong} 
- \int_0^\tau \int_{\mc{F}_1(t)} \vU^s_{\mc{F}}\cdot \frac{\partial}{\partial t}\vphi_{\mc{F}} - \int_0^\tau \int_{\mc{B}_1(t)}\vr_\B \vU^s_{\mc{B}}\cdot \frac{\partial}{\partial t}\vphi_{\mc{B}} - \int_0^\tau \int_{\mc{F}_1(t)}\Big( (  \langle Y_{t,\x},\vu_{1 \mc{F}}\rangle \otimes \vU^s_{\mc{F}}) : \nabla \vphi_{\mc{F}} - (\mathbf{U}^s_\mc{F} - \langle Y_{t,\x},\vu_{1 \mc{F}}\rangle) \cdot \nabla \mathbf{U}^s_\mc{F} \cdot \vphi_{\mc{F}} \Big)\\
- \int_0^\tau\int_{\mc{B}_1(t)} \rho_{\B}\Big( (\mathbf{u}_{1 \mc{B}} \otimes \vU^s_{\mc{B}}) : \nabla \vphi_{\mc{B}} - (\mathbf{U}^s_\mc{B} -\mathbf{u}_{1 \mc{B}})\cdot \nabla \mathbf{U}^s_\mc{B} \cdot \vphi_{\mc{B}} \Big)
= \int_{0}^{\tau}\int_{\mc{F}_1(t)}\big (%
\mathbf{F}\cdot
\vphi_{\mc{F}}\big )\\
- \int_{0}^{\tau}((\vw^s-\vw_1)\times (\tn{J}_1 \vw^s)\cdot \vphi_{\mathcal{B},\mathbf{w}}+m(\vw^s-\vw_1)\times \mathbf{V}^s\cdot \vphi_{\mathcal{B},\mathbf{V}}) 
+ \int_{\mc{F}_0} (\vU^s_{\mc{F}}\cdot \vphi_{\mc{F}})(0)
- \int_{\mc{F}_1(\tau)} (\vU^s_{\mc{F}}\cdot \vphi_{\mc{F}})(\tau)\\
+ \int_{\mc{B}_0} (\vr_\B\vU^s_{\mc{B}}\cdot \vphi_{\mc{B}})(0)
- \int_{\mc{B}_1(\tau)} (\vr_\B\vU^s_{\mc{B}}\cdot \vphi_{\mc{B}})(\tau),
\end{multline}
where the test function $\vphi\in V_T$ with $\vphi_{\mc{B}}$ is rigid on $\mc{B}_1(t)$, i.e,
\begin{equation*}
  \vphi_{\mc{B}}(t,\x)= \vphi_{\mathcal{B},\mathbf{V}} (t) + \vphi_{\mathcal{B},\mathbf{w}} (t) \times (\x-\X_1(t)),\quad \mbox{ for }\x\in \mc{B}_1(t) .
\end{equation*}
Observe that by applying the Reynolds transport \cref{t:rey}, we have
\begin{equation} \label{strong-2}
- \int_0^\tau \int_{\mc{F}_1(t)} \vU^s_{\mc{F}}\cdot \frac{\partial}{\partial t}\vU^s_{\mc{F}} 
= 
\frac{1}{2}\int_{\mc{F}_0} (\vU^s_{\mc{F}}\cdot \vU^s_{\mc{F}})(0)
- \frac{1}{2}\int_{\mc{F}_1(\tau)} (\vU^s_{\mc{F}}\cdot \vU^s_{\mc{F}})(\tau)
+ \int_0^\tau \int_{\mc{F}_1(\tau)}(\langle Y_{t,\x},\vu_{1 \mc{F}}\rangle \otimes \vU^s_{\mc{F}}) : \nabla \vU^s_{\mc{F}}  ,
\end{equation}
\begin{equation} \label{strong-body}
- \int_0^\tau \int_{\mc{B}_1(t)} \vr_\B \vU^s_{\mc{B}}\cdot \frac{\partial}{\partial t}\vU^s_{\mc{B}} 
= 
\frac{1}{2}\int_{\mc{B}_0} (\vr_\B \vU^s_{\mc{B}}\cdot \vU^s_{\mc{B}})(0)
- \frac{1}{2}\int_{\mc{B}_1(\tau)} (\vr_\B \vU^s_{\mc{B}}\cdot \vU^s_{\mc{B}})(\tau)
+ \int_0^\tau \int_{\mc{B}_1(t)} \Big( \vr_\B(\mathbf{u}_{1 \mc{B}} \otimes \vU^s_{\mc{B}}) : \nabla \vU^s_{\mc{B}}  \Big).
\end{equation}
Consequently, by taking the test function $\vphi_{\mc{F}}= \vU^s_{\mc{F}}$, $\vphi_{\mc{B}}=\vU^s_{\mc{B}}$ in \eqref{strong} and by using \eqref{strong-2}-\eqref{strong-body}, we obtain the following energy equality: for a.e. $\tau \in [0,T]$
\begin{multline}\label{energy:strong}
\int_{\mc{F}_1(\tau)} \frac{1}{2}|\vU^s_{\mc{F}}|^2 + \int_{\mc{B}_1(\tau)} \frac{1}{2}\vr_\B |\vU^s_{\mc{B}}|^2
+  \int_0^\tau \int_{\mc{F}_1(t)}\Big(  (\mathbf{U}^s_\mc{F} - \langle Y_{t,\x},\vu_{1 \mc{F}}\rangle)\cdot \nabla \mathbf{U}^s_\mc{F} \cdot \vU^s_{\mc{F}} \Big) 
+ \int_0^\tau \int_{\mc{B}_1(t)}\vr_\B \Big(  (\mathbf{U}^s_\mc{B} -  \mathbf{u_{1 \mc{B}}})\cdot \nabla \mathbf{U}^s_\mc{B} \cdot \vU^s_{\mc{B}}) \Big)\\
=\int_{0}^{\tau}\int_{\mc{F}_1(t)}\big (
\mathbf{F}\cdot
\vU^s_{\mc{F}}\big )
- \int_{0}^{\tau}((\vw^s-\vw_1)\times (\tn{J}_1 \vw^s)\cdot \mathbf{w}^s + m(\vw^s-\vw_1)\times \mathbf{V}^s\cdot \mathbf{V}^s) 
+ \int_{\mc{F}_0} \frac{1}{2}|\mathbf{u}_{\mc{F}_0}|^2 + \frac m2 |\vV_0|^2 + \frac 12\mathbb{J}(0)\vw_0\cdot\vw_0
\end{multline}

\subsection{Proof of Theorem \ref{thm:main result}}

\bigskip
We introduce the following energy functional for the measure-valued solution
\begin{equation}\label{energy_funct_weak}
E(Y_{t,\x})(t) = \int_{\mc{F}_1(t)} \langle Y_{t,\x},\frac{1}{2}|\vu_{1 \mc{F}}|^2\rangle + \int_{\mc{B}_1(t)} \frac{1}{2}\vr_\B |\vu_{1 \mc{B}}|^2.
\end{equation}
Similarly, for the transformed strong solution we have
\begin{equation}\label{energy_funct_strong}
E(\vU^s)(t) = \int_{\mc{F}_1(t)} \frac{1}{2}|\vU^s_{\mc{F}}|^2 + \int_{\mc{B}_1(t)} \frac{1}{2}\vr_\B |\vU^s_{\mc{B}}|^2 = \int_{\mc{F}_1(t)} \langle Y_{t,\x},\frac{1}{2}|\vU^s_{\mc{F}}|^2\rangle + \int_{\mc{B}_1(t)} \frac{1}{2}\vr_\B |\vU^s_{\mc{B}}|^2.
\end{equation}
Given (\ref{energy_funct_weak}) and (\ref{energy_funct_strong}), we write a relative energy functional as follows
\begin{multline}\label{rel_en_funct}
\left[\mathcal{E}(Y_{t,\x}|\vU^s)\right]^{t=\tau}_{t=0}
=
\left[\int_{\mc{F}_1} \langle Y_{t,\x},\frac{1}{2}|\vu_{1 \mc{F}}-\vU^s_{\mc{F}}|^2\rangle\right]^{t=\tau}_{t=0} 
+ \left[\int_{\mc{B}_1} \frac{1}{2}\vr_\B |\vu_{1 \mc{B}}-\vU^s_{\mc{B}}|^2\right]^{t=\tau}_{t=0}
\\= \left[E(Y_{t,\x})\right]^{t=\tau}_{t=0} 
+ \left[E(\vU^s)\right]^{t=\tau}_{t=0}
- \left[\int_{\mc{F}_1} \langle Y_{t,\x},\vu_{1 \mc{F}}\rangle\cdot\vU^s_{\mc{F}}\right]^{t=\tau}_{t=0}
- \left[\int_{\mc{B}_1} \vr_\B \vu_{1 \mc{B}}\cdot\vU^s_{\mc{B}}\right]^{t=\tau}_{t=0}
\end{multline}
In particular we emphasize that the integral over the rigid body gives the control
\begin{equation}
    c(|\vV_1 - \vV^s|^2(\tau) + |\vw_1 - \vw^s|^2(\tau)) \leq \mathcal{E}(Y_{t,\x}|\vU^s)(\tau) 
\end{equation}
for some strictly positive $c$ which does not depend on $\tau$.

Now, from the weak formulation (\ref{momentum:euler}), we have

\begin{multline}\label{momentum:euler_ref}
- \int_0^\tau\int_{\mc{F}_1(t)} \langle Y_{t,\x},\vu_{1\mc{F}}\rangle\cdot {\partial_ t}\vU^s_{\mc{F}} - \int_0^\tau\int_{\mc{B}_1(t)} \vr_\B \vu_{1\mc{B}}\cdot {\partial_ t}\vU^s_{\mc{B}} - \int_0^\tau\int_{\mc{F}_1(t)} \langle Y_{t,\x},(\vu_{1\mc{F}} \otimes \vu_{1\mc{F}}) \rangle : \nabla \vU^s_{\mc{F}}
 \\=  \int_{\mc{F}_0}\langle Y_{0,\x}, \vu_{1\mc{F}}\rangle\cdot \vU^s_{\mc{F}}(0)
 - \int_{\mc{F}_1(\tau)}\langle Y_{\tau,\x}, \vu_{1\mc{F}}\rangle\cdot \vU^s_{\mc{F}}(\tau)
 + \int_{\mc{B}_0} (\vr_\B \vu_{1\mc{B}}\cdot \vU^s_{\mc{B}})(0)
 - \int_{\mc{B}_1(\tau)} (\vr_\B \vu_{1\mc{B}}\cdot \vU^s_{\mc{B}})(\tau)
 + \int_{0}^\tau \langle \mu^M_{D}, \nabla \vU^s_{\mc{F}}\rangle,
\end{multline}
where we used as a test function the strong solution $\vU^s$. Relation (\ref{momentum:euler_ref}) could be written in a more concise form as
\begin{multline}\label{momentum:euler_ref_2}
\left[\int_{\mc{F}_1(\cdot)}\langle Y_{t,\x},\vu_{1\mc{F}}\rangle\cdot \vU^s_{\mc{F}}\right]^{t=\tau}_{t=0}
+ \left[\int_{\mc{B}_1(\cdot)} \vr_\B \vu_{1\mc{B}}\cdot \vU^s_{\mc{B}}\right]^{t=\tau}_{t=0}
\\=
\int_0^\tau\int_{\mc{F}_1(t)} \langle Y_{t,\x},\vu_{1\mc{F}}\rangle\cdot {\partial_ t}\vU^s_{\mc{F}} +\int_0^\tau\int_{\mc{B}_1(t)} \vr_\B \vu_{1\mc{B}}\cdot {\partial_ t}\vU^s_{\mc{B}}  +\int_0^\tau\int_{\mc{F}_1(t)} \langle Y_{t,\x},(\vu_{1\mc{F}} \otimes \vu_{1\mc{F}}) \rangle : \nabla \vU^s_{\mc{F}}
+ \int_{0}^\tau \langle \mu^M_{D}, \nabla \vU^s_{\mc{F}}\rangle.
\end{multline}
Now, using the energy balance (\ref{energy:euler}) and (\ref{energy:strong}), we have from \eqref{rel_en_funct}:
\begin{multline*}
    \left[\mathcal{E}(Y_{t,\x}|\vU^s)\right]^{t=\tau}_{t=0}
    \leq
    -\mathcal{D}(\tau)
    - \int_0^\tau \int_{\mc{F}_1(t)}\Big( (\mathbf{U}^s_\mc{F} -\langle Y_{t,\x},\mathbf{u}_{1\mc{F}}\rangle)\cdot \nabla \mathbf{U}^s_\mc{F}\cdot \vU^s_{\mc{F}} \Big) - \int_0^\tau \int_{\mc{B}_1(t)}\vr_\B\Big(  (\mathbf{U}^s_\mc{B} -\mathbf{u}_{1\mc{B}})\cdot \nabla \mathbf{U}^s_\mc{B}\cdot \vU^s_{\mc{B}} \Big)
    \\ +\int_{0}^{\tau}\int_{\mc{F}_1(t)}\big (%
    \mathbf{F}\cdot
    \vU^s_{\mc{F}}\big )
    -\int_{0}^{\tau}((\vw^s-\vw_1)\times (\tn{J}_1 \vw^s)\cdot \mathbf{w}^s+m(\vw^s-\vw_1)\times \mathbf{V}^s\cdot \mathbf{V}^s)\\
    - \left[\int_{\mc{F}_1(\cdot)} \langle Y_{t,\x},\vu_{1\mc{F}}\rangle\cdot\vU^s_{\mc{F}}\right]^{t=\tau}_{t=0}
    - \left[\int_{\mc{B}_1(\cdot)} \vr_\B \vu_{1\mc{B}}\cdot\vU^s_{\mc{B}}\right]^{t=\tau}_{t=0}.
\end{multline*}
Namely, using \eqref{momentum:euler_ref_2}, the above relation becomes
\begin{multline} \label{rel-entr-comp2}
    \left[\mathcal{E}(Y_{t,\x}|\vU^s)\right]^{t=\tau}_{t=0} + \mathcal{D}(\tau)
    \leq
    - \int_0^\tau \int_{\mc{F}_1(t)}\Big( (\mathbf{U}^s_\mc{F} -\langle Y_{t,\x},\mathbf{u}_{1\mc{F}}\rangle)\cdot \nabla \mathbf{U}^s_\mc{F}\cdot\vU^s_{\mc{F}} \Big)
     - \int_0^\tau \int_{\mc{B}_1(t)}\vr_\B\Big(  (\mathbf{U}^s_\mc{B} -\mathbf{u}_{1\mc{B}})\cdot \nabla \mathbf{U}^s_\mc{B}\cdot\vU^s_{\mc{B}} \Big)
    \\ +\int_{0}^{\tau}\int_{\mc{F}_1(t)}\big (
    \mathbf{F}\cdot
    \vU^s_{\mc{F}}\big )
    -\int_{0}^{\tau}((\vw^s-\vw_1)\times (\tn{J}_1 \vw^s)\cdot \mathbf{w}^s+m(\vw^s-\vw_1)\times \mathbf{V}^s\cdot \mathbf{V}^s)
    - \int_0^\tau\int_{\mc{F}_1(t)} \langle Y_{t,\x},\vu_{1\mc{F}}\rangle\cdot {\partial_ t}\vU^s_{\mc{F}}\\
     -\int_0^\tau\int_{\mc{B}(t)} \vr_\B \vu_{1\mc{B}}\cdot {\partial_ t}\vU^s_{\mc{B}}  -\int_0^\tau\int_{\mc{F}_1(t)} \langle Y_{t,\x},(\vu_{1\mc{F}} \otimes \vu_{1\mc{F}}) \rangle : \nabla \vU^s_{\mc{F}}
    - \int_{0}^\tau \langle \mu^M_{D}, \nabla \vU^s_{\mc{F}}\rangle.
\end{multline}
Now observe that, using (\ref{eq:fluidtransformed}), we have
\begin{multline} \label{q2}
    - \int_0^\tau\int_{\mc{F}_1(t)} \langle Y_{t,\x},\vu_{1\mc{F}}\rangle\cdot {\partial_ t}\vU^s_{\mc{F}}
    =
    \int_0^\tau\int_{\mc{F}_1(t)} \left( \dv \left( \bU^s_{\F} \otimes \bU^s_{\F} \right) \cdot \langle Y_{t,\x},\vu_{1\mc{F}}\rangle
    - \mathbf{F} \cdot \langle Y_{t,\x},\vu_{1\mc{F}}\rangle \right) + \int_0^\tau\int_{\partial \mc{B}_1(t)} (\vu_{1\mc{F}}\cdot \vc{n}) P^s_{\mc{F}} \\
= \int_0^\tau\int_{\mc{F}_1(t)} \left( \dv \left( \bU^s_{\F} \otimes \bU^s_{\F} \right) \cdot \langle Y_{t,\x},\vu_{1\mc{F}}\rangle
    - \mathbf{F} \cdot \langle Y_{t,\x},\vu_{1\mc{F}}\rangle \right) + \int_0^\tau\int_{\partial \mc{B}_1(t)} (\vu_{1\mc{B}}\cdot \vc{n}) P^s_{\mc{F}} ,
\end{multline}
where we integrated by parts the pressure term in (\ref{eq:fluidtransformed}), and use the divergence-free condition with the boundary condition for the measure-valued solution.
Consequently, using the compatibility condition (\ref{cc:euler}) and the equality \eqref{q2}, the
 relation \eqref{rel-entr-comp2} reads
\begin{multline} \label{rel-entr-comp4}
   \left[\mathcal{E}(Y_{t,\x}|\vU^s)\right]^{t=\tau}_{t=0}
    +\mathcal{D}(\tau)
    \leq
    C\int_{0}^\tau 
    \mathcal{D}(\cdot)
    -\int_0^\tau \int_{\mc{F}_1(t)}\Big(  (\mathbf{U}^s_\mc{F} -\langle Y_{t,\x},\mathbf{u}_{1\mc{F}}\rangle)\cdot \nabla \mathbf{U}^s_\mc{F}\cdot\vU^s_{\mc{F}} \Big)- \int_0^\tau \int_{\mc{B}_1(t)}\vr_\B\Big(  (\mathbf{U}^s_\mc{B} -\mathbf{u}_{1\mc{B}})\cdot \nabla \mathbf{U}^s_\mc{B}\cdot\vU^s_{\mc{B}} \Big)
    \\ +\int_{0}^{\tau}\int_{\mc{F}_1(t)}
    \mathbf{F}\cdot
    \left( \vU^s_{\mc{F}} - \langle Y_{t,\x},\vu_{1\mc{F}}\rangle \right)
    -\int_{0}^{\tau}((\vw^s-\vw_1)\times (\tn{J}_1 \vw^s)\cdot \mathbf{w}^s+m(\vw^s-\vw_1)\times \mathbf{V}^s\cdot \mathbf{V}^s)
    -\int_0^\tau\int_{\mc{B}(t)} \vr_\B \vu_{1\mc{B}}\cdot {\partial_ t}\vU^s_{\mc{B}} 
    \\ 
 + \int_0^\tau\int_{\partial \mc{B}_1(t)} (\vu_{1\mc{B}}\cdot \vc{n}) P^s_{\mc{F}}
   +\int_0^\tau\int_{\mc{F}_1(t)} \dv \left( \bU^s_{\F} \otimes \bU^s_{\F} \right) \cdot \langle Y_{t,\x},\vu_{1\mc{F}}\rangle    - \int_0^\tau\int_{\mc{F}_1(t)} \langle Y_{t,\x},(\vu_{1\mc{F}} \otimes \vu_{1\mc{F}}) \rangle : \nabla \vU^s_{\mc{F}} .
\end{multline}
Now, we analyze the last two terms in (\ref{rel-entr-comp4}). We have
\begin{equation*}
    \int_0^\tau\int_{\mc{F}_1(t)} \dv \left( \bU^s_{\F} \otimes \bU^s_{\F} \right) \cdot \langle Y_{t,\x},\vu_{1\mc{F}}\rangle
    - \int_0^\tau\int_{\mc{F}_1(t)} \langle Y_{t,\x},(\vu_{1\mc{F}} \otimes \vu_{1\mc{F}}) \rangle : \nabla \vU^s_{\mc{F}} 
\end{equation*}
\begin{equation*}
    =  \int_0^\tau\int_{\mc{F}_1(t)} (\bU^s_{\F} \cdot \nabla \bU^s_{\F} ) \cdot \langle Y_{t,\x},\vu_{1\mc{F}}\rangle
    - \int_0^\tau\int_{\mc{F}_1(t)} \langle Y_{t,\x},\vu_{1\mc{F}}\rangle \cdot \nabla \vU^s_{\mc{F}} \langle Y_{t,\x},\vu_{1\mc{F}}\rangle
\end{equation*}
\begin{equation*}
=  \int_0^\tau\int_{\mc{F}_1(t)} \langle Y_{t,\x},\vu_{1\mc{F}}\rangle \cdot \nabla \vU^s_{\mc{F}}\Big(\bU^s_{\F}-\langle Y_{t,\x},\vu_{1\mc{F}}\rangle\Big).
\end{equation*}
Thus \eqref{rel-entr-comp4} can be rewritten as:
\begin{multline} \label{rewrite-rel-entr}
    \left[\mathcal{E}(Y_{t,\x}|\vU^s)\right]^{t=\tau}_{t=0}
    +\mathcal{D}(\tau)
    \leq
    C\int_{0}^\tau 
    \mathcal{D}(\cdot)
    -\int_0^\tau \int_{\mc{F}_1(t)}\Big(  (\mathbf{U}^s_\mc{F} -\langle Y_{t,\x},\mathbf{u}_{1\mc{F}}\rangle)\cdot \nabla \mathbf{U}^s_\mc{F}\cdot\vU^s_{\mc{F}} \Big)\\
  + \int_0^\tau\int_{\mc{F}_1(t)} \langle Y_{t,\x},\vu_{1\mc{F}}\rangle \cdot \nabla \vU^s_{\mc{F}}\cdot\Big(\bU^s_{\F}-\langle Y_{t,\x},\vu_{1\mc{F}}\rangle\Big)
     +\int_{0}^{\tau}\int_{\mc{F}_1(t)}
    \mathbf{F}\cdot
    \left( \vU^s_{\mc{F}} - \langle Y_{t,\x},\vu_{1\mc{F}}\rangle \right)\\
    - \int_0^\tau \int_{\mc{B}_1(t)}\vr_\B\Big(  (\mathbf{U}^s_\mc{B} -\mathbf{u}_{1\mc{B}})\cdot \nabla \mathbf{U}^s_\mc{B}\cdot\vU^s_{\mc{B}} \Big)
    -\int_0^\tau\int_{\mc{B}_1(t)} \vr_\B \vu_{1\mc{B}}\cdot {\partial_ t}\vU^s_{\mc{B}} 
 + \int_0^\tau\int_{\partial \mc{B}_1(t)} (\vu_{1\mc{B}}\cdot \vc{n}) P^s_{\mc{F}}\\
 -\int_{0}^{\tau}((\vw^s-\vw_1)\times (\tn{J}_1 \vw^s)\cdot \mathbf{w}^s+m(\vw^s-\vw_1)\times \mathbf{V}^s\cdot \mathbf{V}^s) = C\int_{0}^\tau 
    \mathcal{D}(\cdot) + \mc{I}_{\F}^1 + \mc{I}_{\F}^2 + \mc{I}_{\F}^3 + \mc{I}_{\B}^1 + \mc{I}_{\B}^2 + \mc{I}_{\B}^3 + \mc{I}_{\B}^4.
\end{multline}
Now, observe that 
\begin{equation*}
\mc{I}_{\F}^1 + \mc{I}_{\F}^2 = \int_0^\tau \int_{\mc{F}_1(t)}  (\mathbf{U}^s_\mc{F} -\langle Y_{t,\x},\mathbf{u}_{1\mc{F}}\rangle)\cdot \nabla \mathbf{U}^s_\mc{F}\cdot(\mathbf{U}^s_\mc{F} -\langle Y_{t,\x},\mathbf{u}_{1\mc{F}}\rangle)
\end{equation*}
and therefore 
\begin{equation*}
|\mc{I}_{\F}^1 + \mc{I}_{\F}^2| \leq C\int_{0}^\tau \mathcal{E}(\cdot).    
\end{equation*}
The estimate \eqref{est:F} of the forcing term $\mathbf{F}$ in \cref{l:32} gives us
\begin{equation*}
|\mc{I}_{\F}^3| = |\int_{0}^{\tau}\int_{\mc{F}_1(t)}
    \mathbf{F}\cdot
    \left( \vU^s_{\mc{F}} - \langle Y_{t,\x},\vu_{1\mc{F}}\rangle \right)| \leq C\int_{0}^\tau \mathcal{E}(\cdot).
\end{equation*}
Regarding the terms on the rigid body, observe that as $\vu_{1 \B}\cdot\nabla\mathbf{U}^s_\mc{B}\cdot \vu_{1 \B}=(\vu_{1 \B}\otimes \vu_{1 \B}): \nabla\mathbf{U}^s_\mc{B}=0$, we can write
\begin{align*}
&-\Big(  (\mathbf{U}^s_\mc{B} -\mathbf{u}_{1\mc{B}})\cdot \nabla \mathbf{U}^s_\mc{B}\cdot\vU^s_{\mc{B}} \Big)
    - \vu_{1\mc{B}}\cdot {\partial_ t}\vU^s_{\mc{B}} \\&= \mathbf{U}^s_\mc{B}\cdot \nabla \mathbf{U}^s_\mc{B}\cdot \vu_{1 \B} - \mathbf{U}^s_\mc{B}\cdot \nabla \mathbf{U}^s_\mc{B}\cdot \mathbf{U}^s_\mc{B} - \vu_{1 \B}\cdot\nabla\mathbf{U}^s_\mc{B}\cdot \vu_{1 \B} + \vu_{1 \B}\cdot\nabla\mathbf{U}^s_\mc{B}\cdot\mathbf{U}^s_\mc{B}  - \Big( \vU^s_{\mc{B}}\cdot \nabla \vU^s_{\mc{B}} \cdot \mathbf{u}_{1\mc{B}}+ \vu_{1\mc{B}}\cdot {\partial_ t}\vU^s_{\mc{B}} \Big)\\&
 = (\mathbf{U}^s_\mc{B} -\mathbf{u}_{1\mc{B}})\cdot \nabla \mathbf{U}^s_\mc{B}\cdot(\mathbf{u}_{1\mc{B}}- \mathbf{U}^s_\mc{B}) - \Big( \vU^s_{\mc{B}}\cdot \nabla \vU^s_{\mc{B}} \cdot \mathbf{u}_{1\mc{B}}+ \vu_{1\mc{B}}\cdot {\partial_ t}\vU^s_{\mc{B}} \Big).   
\end{align*}

Moreover, using the calculations presented in detail in \cref{17:25}, we have:
 \begin{multline}\label{gulu}
 -\int_0^\tau \int_{\mc{B}_1(t)}\Big(\vr_\B  \vU^s_{\mc{B}}\cdot \nabla \vU^s_{\mc{B}} \cdot \mathbf{u}_{1\mc{B}}+ \vr_\B \vu_{1\mc{B}}\cdot {\partial_ t}\vU^s_{\mc{B}} \Big) = -\int_0^\tau\int_{\partial \mc{B}_1(t)} (\vu_{1\mc{B}}\cdot \vc{n}) P^s_{\mc{F}} + \int_0^{\tau} (m(\vw^s-\vw_1)\times \mathbf{V}^s)\cdot \vV_1\\
 + \int_0^{\tau}((\vw^s-\vw_1)\times (\tn{J}_1 \vw^s))\cdot \vw_1.
 \end{multline}
Thus,
\begin{multline*}
\mc{I}_{\B}^1 + \mc{I}_{\B}^2 + \mc{I}_{\B}^3 + \mc{I}_{\B}^4 = \int_0^\tau\int_{\partial \mc{B}_1(t)} \vr_\B (\mathbf{U}^s_\mc{B} -\mathbf{u}_{1\mc{B}})\cdot \nabla \mathbf{U}^s_\mc{B}\cdot (\mathbf{u}_{1\mc{B}}- \mathbf{U}^s_\mc{B}) + \int_{0}^{\tau}((\vw^s-\vw_1)\times (\tn{J}_1 \vw^s)\cdot (\vw_1-\mathbf{w}^s))\\ +
\int_0^{\tau}m(\vw^s-\vw_1)\times \mathbf{V}^s\cdot (\vV_1-\mathbf{V}^s)).
\end{multline*}
So, 
\begin{equation*}
\Big| \mc{I}_{\B}^1 + \mc{I}_{\B}^2 + \mc{I}_{\B}^3 + \mc{I}_{\B}^4 \Big|
\leq C\int_{0}^\tau \mathcal{E}(\cdot).
\end{equation*}
Hence, \eqref{rewrite-rel-entr} becomes
\begin{equation*} \label{rel-entr-comp5}
   \mathcal{E}(Y_{t,\x}|\vU^s)(\tau)
    +\mathcal{D}(\tau)
    \leq
    C\int_{0}^\tau 
    \Big( \mc{E}(\cdot) + \mathcal{D}(\cdot) \Big) + \mathcal{E}(Y_{t,\x}|\vU^s)(0).
\end{equation*}
Since the initial states coincide and therefore the value of the relative energy functional at time zero is equal to zero, we use Gronwall lemma to conclude that for $t\in (0,T)$:
\begin{equation*}
\mathcal{D} = 0 \mbox{ and } Y_{t,\x}=\delta_{\vU^s_{\mc{F}}} \text { for } \x\in \mc{F}_1, \ \vu_{1\mc{B}}=\vU^s_{\mc{B}}.
\end{equation*}

Finally, we have to show that in fact $\vV_1 = \vV_2$, $\vw_1 = \vw_2$, $\B_1(t) = \B_2(t)$ and $Y_{t,\x} = \delta_{\vu_{2\F}}$ on $\F_1(t)$. This argument is the same as was presented i.e. in \cite{KrNePi_2}. We start with \eqref{eq:strongbody} to write $\vV_1 =\tn{O}_1\tn{O}_2^T\vV_2$ and $\vw_1 = \tn{O}_1\tn{O}_2^T\vw_2$, therefore
\begin{align} \label{eq:121}
    \tn{O}_1^T\vV_1 &= \tn{O}_2^T\vV_2, \\
    \tn{O}_1^T\vw_1 &= \tn{O}_2^T\vw_2. \label{eq:122}
\end{align}
Moreover, we use the relation between the angular velocity of the body and the isometries describing the motion of the body \eqref{eq:QO} to conclude
\begin{equation}
    \vw_i \times \x = \tn{Q}_i\x = \frac{\d \tn{O}_i}{\dt}\tn{O}_i^T\x
\end{equation}
for $i=1,2$ and for all $\x \in \R^3$. Thus
\begin{equation}
    \tn{O}_i^T\vw_i \times \x = \tn{O}_i^T(\vw_i \times \tn{O}_i\x) = \tn{O}_i^T\frac{\d \tn{O}_i}{\dt}\tn{O}_i^T\tn{O}_i\x = \tn{O}_i^T\frac{\d \tn{O}_i}{\dt}\x
\end{equation}
for $i=1,2$ and all $\x \in \R^3$ and finally combining this with \eqref{eq:122} we deduce
\begin{equation} \label{eq:458}
    \tn{O}_1^T\frac{\d \tn{O}_1}{\dt} = \tn{O}_2^T\frac{\d \tn{O}_2}{\dt}.
\end{equation}
We rewrite \eqref{eq:458} as
\begin{equation*}
    \frac{\d(\tn{O}_1 - \tn{O}_2)}{\dt} = (\tn{O}_1-\tn{O}_2)\tn{O}_2^T\frac{\d \tn{O}_2}{\dt}.
\end{equation*}
Denoting $\tn{O}_\Delta = \tn{O}_1 - \tn{O}_2$ and treating $\tn{O}_2$ as a given function of time we end up with
\begin{align} \label{eq:ODE1}
    \frac{\d \tn{O}_\Delta(t)}{\dt} &= \tn{O}_\Delta(t) \tn{W}(t), \\
    \tn{O}_\Delta(0) &= 0, \label{eq:ODE2}
\end{align}
for some given matrix valued function $\tn{W}(t)$. The problem \eqref{eq:ODE1}-\eqref{eq:ODE2} has a unique solution $\tn{O}_\Delta(t) = 0$, therefore $\tn{O}_1 = \tn{O}_2$ and taking into account \eqref{eq:121} and \eqref{eq:122} we finally end up with $\vV_1 = \vV_2$ and $\vw_1 = \vw_2$. This also proves that the position of the bodies of the measure-valued and strong solutions are the same, i.e. $\B_1(t) = \B_2(t)$.

Finally we can set the cutoff functions $\zeta_i(t,\x)$ ($i=1,2$) introduced in Section \ref{s:cc} to coincide in the case that $\B_1 = \B_2$. Therefore $\Lambda_1(t,\x) = \Lambda_2(t,\x)$, $\bZ_1 = \bZ_2$ and $\ZZ(t,\x) = \ZZj(t,\x) = \x$ for all $\x \in \Omega$ and $t \in (0,T)$. Finally we use \eqref{newsol2} to conclude that $\vU^s_{\F} = \vu_{2\F}$. 
This also proves that $T_{min}$ has to be equal to $T_0$.

\section{Appendix}\label{s:appendix}

\subsection{Reynolds transport theorem}

For completeness of presentation we state here the Reynolds transport theorem which is one of the key ingredients in analysis of fluid on moving domains.

\begin{Theorem}\label{t:rey}
Let $f$ be a function such that all integrals in the formula below are well defined. Let $\vu_\B$ be a rigid velocity field describing the motion of the body $\B(t)$. Then the following formula for time derivative of an integral over the fluid domain holds.
\begin{equation}
    \frac{d}{dt}\int_{\F(t)} f \dx = \int_{\F(t)} \de_t f \dx + \int_{\de \B(t)} f \vu_\B \cdot \vn\, \dS
\end{equation}
\end{Theorem}

\subsection{Weak formulation of the momentum equation}\label{App-Euler}
In this section, we present the calculation behind the weak formulation of the momentum equation \eqref{momentum:euler}.

Let all the functions be sufficiently smooth so that we can do integration by parts.
We multiply equation \eqref{fluidmotion:ineuler}$_1$ by test function $\vphi\in V_T$ and integrate over $\mc{F}(t)$ to obtain
\begin{equation}\label{express1}
\frac{d}{dt} \int_{\mc{F}(t)} \vu_{\mc{F}}\cdot \vphi_{\mc{F}} - \int_{\mc{F}(t)} \vu_{\mc{F}}\cdot \frac{\partial}{\partial t}\vphi_{\mc{F}} - \int_{\mc{F}(t)} (\vu_{\mc{F}} \otimes \vu_{\mc{F}}) : \nabla \vphi_{\mc{F}}
= -\int_{\partial \Omega} p_{\mc{F}}\mathbb{I} \vc{n}\cdot \vphi_{\mc{F}} - \int_{\partial \mc{B}(t)} p_{\mc{F}}\mathbb{I} \vc{n}\cdot \vphi_{\mc{F}}.
\end{equation}
As $\vphi \in V_T$, we have $\vphi_{\mc{F}} \cdot \vc{n}=0$ on $\partial\Omega$. Using the Reynolds transport theorem \ref{t:rey} and the mass transport of the body we have
\begin{multline*}
\frac{d}{dt}\int_{\mc{B}(t)}\vr_\B \vu_{\mathcal{B}}\cdot \vphi_{\mc{B}} = \int_{\mc{B}(t)} \frac{\de}{\de t} (\vr_\B \vu_{\mathcal{B}}\cdot \vphi_{\mc{B}}) + \int_{\mc{B}(t)} \Div \left(\vr_\B\vu_{\mc{B}}(\vu_{\mathcal{B}}\cdot \vphi_{\mc{B}})\right)\\ 
= \int_{\mc{B}(t)} \vr_\B\vu_{\mathcal{B}}\cdot \frac{\partial}{\partial t} \vphi_{\mc{B}} + \int_{\mc{B}(t)} \vr_\B \Big(\frac{\partial}{\partial t} \vu_{\mathcal{B}} + \vu_{\mathcal{B}} \cdot \nabla\vu_{\mathcal{B}}\Big)\cdot \vphi_{\mc{B}} + \int_{\mc{B}(t)} \rho_{\B} (\mathbf{u}_{\mc{B}} \otimes \vu_{\mathcal{B}}) : \nabla \vphi_{\mc{B}}.
\end{multline*}
As $\mathbb{D}(\vphi_{\mc{B}})=0$, we have $(\mathbf{u}_{\mc{B}} \otimes \vu_{\mathcal{B}}) : \nabla \vphi_{\mc{B}}=0$. Using the rigid body equations \eqref{bodymotion:ineuler} and the boundary condition $\vphi_{\F} \cdot \vc{n} = \vphi_{\B} \cdot \vc{n}$ on $\partial
\mathcal{B}(t)$, we obtain
\begin{equation}\label{express4}
- \int_{\partial \mc{B}(t)} p_{\mc{F}}\mathbb{I} \vc{n}\cdot \vphi_{\mc{F}} = -\frac{d}{dt}\int_{\mc{B}(t)}\vr_\B \vu_{\mc{B}}\cdot \vphi_{\mc{B}} + \int_{\mc{B}(t)} \vr_\B\vu_{\mc{B}}\cdot \frac{\partial}{\partial t} \vphi_{\mc{B}} .
\end{equation}
Thus by combining the above relations \eqref{express1}-\eqref{express4} and then integrating in time, we have
\begin{equation*}
- \int_0^\tau \int_{\mc{F}(t)} \vu_{\mc{F}}\cdot \frac{\partial}{\partial t}\vphi_{\mc{F}} - \int_0^\tau \int_{\mc{B}(t)}\vr_\B \vu_{\mc{B}}\cdot \frac{\partial}{\partial t}\vphi_{\mc{B}} - \int_0^\tau \int_{\mc{F}(t)} (\vu_{\mc{F}} \otimes \vu_{\mc{F}}) : \nabla \vphi_{\mc{F}}
\end{equation*}
\begin{equation}\label{weak-momentum}
= \int_{\mc{F}_0} (\vu_{\mc{F}}\cdot \vphi_{\mc{F}})(0) 
- \int_{\mc{F}_\tau} (\vu_{\mc{F}}\cdot \vphi_{\mc{F}})(\tau)
+ \int_{\mc{B}_0} (\vr_\B\vu_{\mc{B}}\cdot \vphi_{\mc{B}})(0)
- \int_{\mc{B}_\tau} (\vr_\B\vu_{\mc{B}}\cdot \vphi_{\mc{B}})(\tau)
\end{equation}
for $\tau \in [0,T]$.
\subsection{Derivation of \eqref{strong}.} \label{ss:43}
We multiply equation \eqref{eq:fluidtransformed}$_1$ by test function $\vphi\in V_T$, integrate over $\mc{F}_1(t)$ and use the Reynolds transport theorem \ref{t:rey} to obtain
\begin{multline}\label{new:express1}
\frac{d}{dt} \int_{\mc{F}_1(t)} \bU^s_{\F}\cdot \vphi_{\mc{F}} - \int_{\mc{F}_1(t)} \bU^s_{\F}\cdot \frac{\partial}{\partial t}\vphi_{\mc{F}} -  \int_{\mc{F}_1(t)}\Big( \langle Y_{t,\x},\vu_{1\mc{F}}\rangle  \otimes \vU^s_{\mc{F}}) : \nabla \vphi_{\mc{F}} - (\mathbf{U}^s_\mc{F} - \langle Y_{t,\x},\vu_{1\mc{F}}\rangle)\cdot \nabla \mathbf{U}^s_\mc{F}\cdot\vphi_{\mc{F}} \Big)\\
= -\int_{\partial \Omega} P^s_{\mc{F}}\mathbb{I} \vc{n}\cdot \vphi_{\mc{F}} - \int_{\partial \mc{B}_1(t)} P^s_{\mc{F}}\mathbb{I} \vc{n}\cdot \vphi_{\mc{F}} + \int_{\mc{F}_1(t)} \mathbf{F}\cdot \vphi_{\mc{F}},
\end{multline}
where we combine the weak formulation of the continuity equation and the boundary condition.
As $\vphi \in V_T$, we have $\vphi_{\mc{F}} \cdot \vc{n}=0$ on $\partial\Omega$. We use the Reynolds transport theorem \ref{t:rey} together with the mass transport of the body to obtain
\begin{multline*}
\frac{d}{dt}\int_{\B_1(t)}\vr_{\B} \bU^s_{\B}\cdot \vphi_{\mc{B}} = \int_{\mc{B}_1(t)} \frac{\de}{\de t} (\vr_{\B} \vU^s_{\mc{B}}\cdot \vphi_{\mc{B}}) + \int_{\mc{B}_1(t)} \Div \left( \vr_{\B}\vu_{1\mc{B}} (\vU^s_{\mc{B}}\cdot \vphi_{\mc{B}})\right)
\\ 
= \int_{\mc{B}_1(t)} \vr_{\B}\vU^s_{\mc{B}}\cdot \frac{\partial}{\partial t} \vphi_{\mc{B}} + \int_{\mc{B}_1(t)} \vr_{\B} \Big(\frac{\partial}{\partial t} \vU^s_{\mc{B}} + \vU^s_{\mc{B}} \cdot \nabla\vU^s_{\mc{B}}\Big)\cdot \vphi_{\mc{B}} + \int_{\mc{B}_1(t)} \rho_{\B}\Big( (\mathbf{u}_{1\mc{B}} \otimes \vU^s_{\mc{B}}) : \nabla \vphi_{\mc{B}} - (\mathbf{U}^s_{\mc{B}} -\mathbf{u}_{1\mc{B}})\cdot \nabla \mathbf{U}^s_{\mc{B}}\cdot\vphi_{\mc{B}} \Big).
\end{multline*}
Using the rigid body equations \eqref{eq:newrigidtransformed1}-\eqref{eq:newrigidtransformed2} and the boundary condition $\vphi_{\F} \cdot \vc{n} = \vphi_{\B} \cdot \vc{n}$ on $\partial
\mathcal{B}_1(t)$, we obtain
\begin{multline}\label{new:express4}
- \int_{\partial \mc{B}_1(t)} P^s_{\mc{F}}\mathbb{I} \vc{n}\cdot \vphi_{\mc{F}} = -\frac{d}{dt}\int_{\mc{B}_1(t)}\vr_\B \vu_{\mc{B}}\cdot \vphi_{\mc{B}} + \int_{\mc{B}_1(t)} \vr_\B\vu_{\mc{B}}\cdot \frac{\partial}{\partial t} \vphi_{\mc{B}} + \int_{\mc{B}_1(t)} \rho_{\B}\Big( (\mathbf{u}_{1\mc{B}} \otimes \vU^s_{\mc{B}}) : \nabla \vphi_{\mc{B}} - (\mathbf{U}^s_\mc{B} -\mathbf{u}_{1\mc{B}})\cdot \nabla \mathbf{U}^s_\mc{B}\cdot\vphi_{\mc{B}} \Big)\\
 - ((\vw^s-\vw_1)\times (\tn{J}_1 \vw^s)\cdot \vphi_{\mathcal{B},\mathbf{w}}+m(\vw^s-\vw_1)\times \mathbf{V}^s\cdot \vphi_{\mathcal{B},\mathbf{V}}).
\end{multline}
Thus by combining the above relations \eqref{new:express1}-\eqref{new:express4} and then integrating in time, we have for $\tau \in [0,T]$:
\begin{multline*} 
- \int_0^\tau \int_{\mc{F}_1(t)} \vU^s_{\mc{F}}\cdot \frac{\partial}{\partial t}\vphi_{\mc{F}} - \int_0^\tau \int_{\mc{B}_1(t)}\vr_\B \vU^s_{\mc{B}}\cdot \frac{\partial}{\partial t}\vphi_{\mc{B}} - \int_0^\tau \int_{\mc{F}_1(t)}\Big( ( \langle Y_{t,\x},\vu_{1\mc{F}}\rangle  \otimes \vU^s_{\mc{F}}) : \nabla \vphi_{\mc{F}} - (\mathbf{U}^s_\mc{F} - \langle Y_{t,\x},\vu_{1\mc{F}}\rangle)\cdot \nabla \mathbf{U}^s_\mc{F}\cdot\vphi_{\mc{F}} \Big)\\
- \int_0^\tau\int_{\mc{B}_1(t)} \rho_{\B}\Big( (\mathbf{u}_{1\mc{B}} \otimes \vU^s_{\mc{B}}) : \nabla \vphi_{\mc{B}} - (\mathbf{U}^s_\mc{B} -\mathbf{u}_{1\mc{B}})\cdot \nabla \mathbf{U}^s_\mc{B}\cdot\vphi_{\mc{B}} \Big)
= \int_{0}^{\tau}\int_{\mc{F}_1(t)}\big (
\mathbf{F}\cdot
\vphi_{\mc{F}}\big )\\
- \int_{0}^{\tau}((\vw^s-\vw_1)\times (\tn{J}_1 \vw^s)\cdot \vphi_{\mathcal{B},\mathbf{w}}+m(\vw^s-\vw_1)\times \mathbf{V}^s\cdot \vphi_{\mathcal{B},\mathbf{V}}) 
+ \int_{\mc{F}_0} (\vU^s_{\mc{F}}\cdot \vphi_{\mc{F}})(0)
- \int_{\mc{F}_1(\tau)} (\vU^s_{\mc{F}}\cdot \vphi_{\mc{F}})(\tau)\\
+ \int_{\mc{B}_0} (\vr_\B\vU^s_{\mc{B}}\cdot \vphi_{\mc{B}})(0)
- \int_{\mc{B}_1(\tau)} (\vr_\B\vU^s_{\mc{B}}\cdot \vphi_{\mc{B}})(\tau).
\end{multline*}

\subsection{Derivation of \eqref{gulu}.}\label{17:25}

In this section, we follow the calculations of \cite[Appendix]{KrNePi} to derive our desired identity. We know that
\begin{align*}
\vu_{1 \B}(t,x) &= \vV_1(t) + \vw_1(t) \times (\x - \X_1(t)) \qquad \text{ in } Q_{\B_1} \\
\bU^s_{\B}(t,x) &= \vV^s(t) + \vw^s(t) \times (\x - \X_1(t)) \qquad \text{ in } Q_{\B_1}.
\end{align*}
 A simple calculation gives
\begin{equation} \label{eq:Ap0}
\frac{\de \bU^s_{\B}}{\de t} + (\bU^s_{\B}\cdot\nabla)\bU^s_{\B} = \frac{\d \vV^s}{\d t} + \frac{\d \vw^s}{\d t}\times (\x-\X_1) + \vw^s \times (\vw^s \times (\x-\X_1)) \quad \text{ in } Q_{\B_1}.
\end{equation}
We use \eqref{eq:newrigidtransformed1} to deduce
\begin{equation}\label{eq:Ap1}
    \int_{\B_1(t)} \vr_\B \frac{\d \vV^s}{\d t}\cdot \vV_1  = \int_{\de \B_1(t)}  P^s_{\F}\vc{n}\cdot \vV_1 - (m(\vw^s-\vw_1)\times \mathbf{V}^s)\cdot \vV_1.
\end{equation}
Next, by \eqref{def:X} we get
\begin{align}
    &\int_{\B_1(t)} \vr_\B \frac{\d \vV^s}{\d t}\cdot (\vw_1\times (\x-\X_1))  = \left(\frac{\d \vV^s}{\d t} \times \vw_1\right)\cdot \int_{\B_1(t)} \vr_\B(\x-\X_1)  = 0,\\
    &\int_{\B_1(t)} \vr_\B \left(\frac{\d \vw^s}{\d t} \times (\x-\X_1)\right)\cdot \vV_1  = \vV_1\cdot \left(\frac{\d \vw^s}{\d t} \times \int_{\B_1(t)} \vr_\B(\x-\X_1) \right) = 0, \\
    &\int_{\B_1(t)} \vr_\B (\vw^s\times(\vw^s\times(\x-\X_1))\cdot \vV_1  = \vV_1\cdot \left(\vw^s\times \left(\vw^s\times \int_{\B_1(t)} \vr_\B(\x-\X_1) \right)\right) = 0.
\end{align}
 Using \eqref{def:J} and \eqref{eq:newrigidtransformed2} we obtain
\begin{align}
    &\int_{\B_1(t)} \vr_\B \left(\frac{\d \vw^s}{\d t} \times (\x-\X_1)\right)\cdot(\vw_1\times (\x-\X_1))   = \tn{J}_1\frac{\d \vw^s}{\d t}\cdot \vw_1 \\ \nonumber
    & \quad = (\tn{J}_1\vw^s \times \vw^s)\cdot \vw_1 - ((\vw^s-\vw_1)\times (\tn{J}_1 \vw^s))\cdot \vw_1 + \vw_1\cdot \int_{\de \B_1(t)} (\x-\X_1)\times  P^s_{\F}\vc{n} \\ \nonumber
    & \quad = \tn{J}_1\vw^s \cdot (\vw^s\times\vw_1) - ((\vw^s-\vw_1)\times (\tn{J}_1 \vw^s))\cdot \vw_1 + \int_{\de \B_1(t)} P^s_{\F}\vc{n} \cdot (\vw_1 \times (\x-\X_1)).
\end{align}
We also have
\begin{equation}\label{eq:Ap6}
    \int_{\B_1(t)} \vr_\B\left(\vw^s\times \left(\vw^s \times (\x-\X_1)\right)\right)\cdot(\vw_1\times (\x-\X_1))  = - \tn{J}_1\vw^s\cdot (\vw^s\times\vw_1).
\end{equation}
Summing \eqref{eq:Ap1}-\eqref{eq:Ap6} and using \eqref{eq:Ap0} we end up with
\begin{multline*}
 \int_0^\tau \int_{\mc{B}_1(t)}\Big(\vr_\B  \vU^s_{\mc{B}}\cdot \nabla \vU^s_{\mc{B}} \cdot \mathbf{u}_{1\mc{B}}+ \vr_\B \vu_{1\mc{B}}\cdot {\partial_ t}\vU^s_{\mc{B}} \Big) = \int_0^\tau\int_{\partial \mc{B}_1(t)} (\vu_{1\mc{B}}\cdot \vc{n}) P^s_{\mc{F}} - \int_0^{\tau} (m(\vw^s-\vw_1)\times \mathbf{V}^s)\cdot \vV_1\\
 - \int_0^{\tau}((\vw^s-\vw_1)\times (\tn{J}_1 \vw^s))\cdot \vw_1.
\end{multline*}

\section*{Acknowledgements}

M. C. has been supported partially by the Institute of Mathematics, CAS, and partially by the Croatian Science Foundation under the project MultiFM IP-2019-04-1140. O. K., \v S. N. and A. R. have been supported by the Czech Science Foundation (GA\v CR) project GA19-04243S. The Institute of Mathematics, CAS is supported by RVO:67985840. The research of T.T. is supported by the NSFC Grant No. 11801138.

\end{document}